\documentclass[12pt,reqno]{amsart}
\usepackage{amsmath,amssymb}

\usepackage{hyperref}

\usepackage{graphicx,psfrag}

\newcommand{\R}{\mathbb{R}}
\newcommand{\Z}{\mathbb{Z}}

\newcommand{\Cubn}{C_{u,b}(\mathbb{R}^n)}
 
\newcommand{\Con}{C_0(\mathbb{R}^n)}
\newcommand{\Co}{C_0(\mathbb{R})}
\newcommand{\eps}{\varepsilon}

\def\Xint#1{\mathchoice
{\XXint\displaystyle\textstyle{#1}}
{\XXint\textstyle\scriptstyle{#1}}
{\XXint\scriptstyle\scriptscriptstyle{#1}}
{\XXint\scriptscriptstyle\scriptscriptstyle{#1}}
\!\int}
\def\XXint#1#2#3{{\setbox0=\hbox{$#1{#2#3}{\int}$ }
\vcenter{\hbox{$#2#3$ }}\kern-.6\wd0}}

\def\dashint{\Xint-}

\def\di{\displaystyle}

\newtheorem{theorem}{Theorem}[section]
\newtheorem{lemma}[theorem]{Lemma}
\newtheorem{corollary}[theorem]{Corollary}
\newtheorem{proposition}[theorem]{Proposition}

\newtheorem{remark}[theorem]{Remark}

\numberwithin{equation}{section}

\begin{document}

\title[Fisher-KPP equations with fractional diffusion]
{The influence of fractional diffusion in Fisher-KPP equations}

\thanks{The first author was supported by grants MTM2008-06349-C03-01, 
MTM2011-27739-C04-01 (Spain) and 2009SGR345 (Catalunya). 
The second author is supported by the ANR grant PREFERED}

\author{Xavier Cabr{\'e}}
\address{ICREA and Universitat Polit{\`e}cnica de Catalunya,
Departament de Matem\`atica Aplicada I, Diagonal 647, 08028
Barcelona, Spain}
\email{xavier.cabre@upc.edu}

\author{Jean-Michel Roquejoffre}
\address{Institut de Math\'ematiques, Universit\'e de Toulouse et CNRS (UMR 
5219),
118 route de Narbonne, 31062 Toulouse, France}
\email{roque@mip.ups-tlse.fr}

\begin{abstract}
We study the Fisher-KPP equation where the Laplacian is replaced by the
generator of a Feller semigroup with power decaying kernel, an important 
example being the fractional Laplacian. 
In contrast with the case of the standard Laplacian
where the stable state invades the unstable one at constant speed, we prove that with
fractional diffusion, generated for instance by a stable L\'evy process,
the front position is exponential in time. 
Our results provide a mathematically rigorous justification of numerous 
heuristics about this model.
\end{abstract}

\maketitle

\section{Introduction}

Let $f$ be a function satisfying
\begin{equation}\label{nonl}
f\in C^1([0,1])\text{ is concave,} \ \ f(0)=f(1)=0,\ \text{ and } f'(1)<0<f'(0).
\end{equation}
We may take for instance $f(u)=u(1-u)$. We are interested in  
the large time behavior of solutions $u=u(t,x)$ to the Cauchy problem
\begin{equation} \label{N}
\left \{ \begin{array}{rcll}
u_t + A u & = & f(u) &\ \ \ \ \text{in} \ 
(0,+\infty)\times \mathbb{R}^n , \\
u(0,\cdot)& = & u_0  &\ \ \ \ \text{in} \ \mathbb{R}^n  , 
\qquad 0\leq u_0 \leq 1,
\end{array}\right.
\end{equation}
where $A$ is the infinitesimal generator of a Feller 
semigroup. Important examples are $A=-\Delta$ (the classical Laplacian) and 
$A=(-\Delta)^\alpha$ with
$\alpha\in (0,1)$ (the fractional Laplacian). 
Given $\lambda\in (0,1)$, we want to describe how the level sets 
$\{x\in \mathbb{R}^n \, :\, u(t,x)=\lambda\}$
spread as time goes to $+\infty$.

When $A=-\Delta$ is the standard Laplacian, the equation becomes
\begin{equation}
\label{e2.1}
u_t-\Delta u=f(u)\ \ \ \ \text{in} \ 
(0,+\infty)\times \mathbb{R}^n
\end{equation}
and the following result of Aronson and Weinberger \cite{AW} describes
the evolution of compactly supported data.

\begin{theorem}[\cite{AW}]\label{AW78}
Let $u$ be a solution of \eqref{e2.1} with $u(0,\cdot)\not\equiv 0$ compactly supported 
in~$\R^n$ and satisfying $0\leq u(0,\cdot)\leq 1$. 
Let $c_\ast=2\sqrt{f'(0)}$.  Then,
\begin{enumerate}
\item[a)]  if $c>c_\ast$, then $u(t,x)\to 0$ uniformly in $\{|x|\geq ct\}$ as $t\to+\infty$.
\item[b)] if $c<c_\ast$, then $u(t,x)\to 1$ uniformly in $\{|x|\leq ct\}$ as $t\to+\infty$.
\end{enumerate}
\end{theorem}
\noindent
In addition, \eqref{e2.1} admits planar traveling wave solutions connecting 0 and 1, that is,
solutions of the form $u(t,x)=\phi(x\cdot e+ct)$ with
\begin{equation}\label{trav-wave-1}
-\phi''+c\phi'=f(\phi)\quad \text{in } \R, \ \ \ \phi(-\infty)=0,\ \phi(+\infty)=1. 
\end{equation}
The constant $c_{\ast}$ in Theorem~\ref{AW78} is the smallest possible speed $c$ in
\eqref{trav-wave-1} for a planar traveling wave to exist. In addition,
Komogorov, Petrovskii, and Piskunov~\cite{KPP} showed that the solution of
\eqref{e2.1} for $n=1$ and with initial datum the Heaviside function 
$H(x)=\chi_{(0,\infty)}(x)$ converges as $t\to +\infty$ to a traveling wave
with speed $c=c_{\ast}$.

Our results, already announced in \cite{CR}, 
show that this situation changes drastically as soon as the Laplacian is replaced
for instance by the fractional Laplacian $(-\Delta)^\alpha$ with $\alpha\in (0,1)$. The equation
then becomes
\begin{equation}\label{eq-frac-lapl}
u_t+(-\Delta)^\alpha u=f(u)\ \ \ \ \text{in} \ 
(0,+\infty)\times \mathbb{R}^n.
\end{equation}
Solutions for the standard heat equation correspond to expected values for
particles moving under a Brownian process. Instead, for $\alpha\in (0,1)$,
the fractional Laplacian is the generator for a stable L\'evy process
---a jump process. It is reasonable to expect
that the existence of jumps (or flights) in the diffusion process will accelerate
the invasion of the unstable state $u=0$ by the stable one, $u=1$.
This has been sustained in the literature, see \cite{MVV,Cas,Cas3} among others,  
through the linearization of the equation at the leading edge of the front, as well as 
through numerical simulations. These heuristics predict that the front position will be
exponential in time ---in contrast with the classical case where it is linear in time
by Theorem~\ref{AW78}. 
The purpose of our work is to provide a rigorous mathematical justification of this fact, and to give an accurate
localisation of the level sets of $u$ in the particular case $\alpha=1/2$ and $f(u)=u-u^2$. In particular, the leading edge analysis is 
not accurate enough.

Reaction equations with fractional diffusion appear in
physical models ---for instance of turbulence, plasmas, and flames--- when
the diffusive phenomena are not properly described by Gaussian (that is, Brownian) processes.
See for example \cite{MVV} for a description of some of these models. 
Equation \eqref{eq-frac-lapl} also appears 
in population dynamics, where it can be obtained in a certain space-time regime
as the asymptotic of an integro-differential model; see \cite{BRR}.
The classical heat equation \eqref{e2.1} can be obtained from the same asymptotic model in a
different space-time regime; see \cite{KPP}.

We consider a larger class of operators than fractional Laplacians.
We are given a continuous function $p=p(t,x)$, with $t>0$ and $x\in\R^n$, such that
\begin{eqnarray}
\label{P1}
&& \hspace{-1.1cm} \bullet\quad  
0<p\in C((0,+\infty)\times\R^n)
\text{ and } \int_{\R^n} p(t,x)\,dx=1 \text{ for all } t>0.\\
\label{P2}
&&  \hspace{-1.1cm} \bullet\quad  
p(t,\cdot )\ast p(s,\cdot )=p(t+s,\cdot) \text{ for all } (s,t)\in(0,\infty)^2.\\
&&  \hspace{-1.1cm} \bullet\quad  
\text{There exist $\alpha\in (0,1)$ and $B>1$ such that, for $t>0$ and } x\in\R^n,\nonumber \\ 
\label{P3}
&&\qquad
\frac{B^{-1}}{t^{\frac{n}{2\alpha}}(1+\vert t^{-{\frac1{2\alpha}}}x\vert^{n+2\alpha})}\leq p(t,x)\leq
\frac{B}{t^{\frac{n}{2\alpha}}(1+\vert t^{-{\frac1{2\alpha}}}x\vert^{n+2\alpha})}.
\end{eqnarray}
We assume no further regularity on $p$ than continuity.
Given a function $u_0\in L^\infty(\R^n)$ and $t>0$, we define
\begin{equation*}
T_tu_0(x):=\left(p(t,\cdot )\ast  u_0\right)(x)=
\int_{\mathbb{R}^n}p
(t,y)u_0(x-y)\, dy.
\end{equation*}
Clearly, the family $T_t$ of bounded linear contractions of $L^\infty(\R^n)$ is a semigroup.
When considered in the Banach space $C_{u,b}(\R^n)$ of uniformly continuous and
bounded functions in $\R^n$, the semigroup is a strongly continuous semigroup
(also called a $C_0$ semigroup) 
and therefore admits an infinitesimal generator
$-A$, defined by  
\begin{equation*}
-Au= \displaystyle\lim_{t\downarrow 0}{\frac{T_tu-u}{t}} 
\end{equation*}
for those $u\in C_{u,b}(\R^n)$ for which the limit exists in the uniform convergence norm. 
The subspace of such functions is called the domain of $A$ and denoted by $D(A)$.
Since the semigroup is strongly continuous, it is well known that  $D(A)$
is a dense subspace of $C_{u,b}(\R^n)$.

Given $u_0\in L^\infty(\R^n)$ the function $u=u(t,x):= T_tu_0 (x)$ is the mild solution 
(see section~2) of the evolution problem
\begin{equation*}
\left \{ \begin{array}{rcll}
u_t + A u & = & 0 &\ \ \ \ \text{in} \ 
(0,+\infty)\times \mathbb{R}^n , \\
u(0,\cdot)& = & u_0  &\ \ \ \ \text{in} \ \mathbb{R}^n . 
\end{array}\right.
\end{equation*}
The function $p$ is called the kernel of the semigroup; it is also called 
the transition probability function.
The operator $A$ is said to be the infinitesimal generator of a Feller semigroup
---since $0\leq u_0 \leq 1$ leads to $0\leq T_tu_0\leq 1$. This property
will lead to a maximum principle for $A$.

The power decay assumption in \eqref{P3} will be crucial for the results of this paper.
The assumption in \eqref{P3} concerning the dependence of the bound on $t^{-\frac{1}{2\alpha}}x$
is related with the self-similarity or 
scale invariance of the underlying Markov process ---an hypothesis often
called ``stability''. Indeed, if one assumes that 
$$
p(t,x)= a(t)^{-n}p(1,a(t)^{-1}x)
$$
for some function $a=a(t)$ and for all $t>0$,
then there exists a constant $\alpha\in (0,1]$ such that 
$a(t)=t^{\frac{1}{2\alpha}}$ ---as in \eqref{P3}; see \cite{Lamp}.

When $A=(-\Delta)^\alpha$ is the fractional Laplacian and $p=p_\alpha$, defined for  
$0<\alpha<1$ as follows, all assumptions \eqref{P1}, \eqref{P2}, and \eqref{P3} are
satisfied. 
If $u\in C^2(\R^n)$ has sufficiently slow growth at infinity ---for instance
$|u(x)|\le C(1+|x|^\gamma)$ 
with $\gamma<2\alpha$--- then
$$
(-\Delta)^\alpha u(x)=C_{n,\alpha}\, P.V.\, \int_{\R^n}\frac{u(x)-u(y)}{\vert x-y\vert^{n+2\alpha}}\, 
dy,
$$
where $P.V.$ stands for principal value and the
constant $C_{n,\alpha}$ is adjusted for the symbol of $(-\Delta)^\alpha$ to be
$\vert\xi\vert^{2\alpha}$. 
Its transition probability function $p$ satisfies
\begin{equation*}
\left \{ \begin{array}{l}
p(t,x) = p_{\alpha}(t,x)=t^{-\frac{n}{2\alpha}}p_{\alpha}
(1,t^{-\frac{1}{2\alpha}}x),  \vspace{2mm}\\
\lim_{|y| \to \infty}{|y|^{n+2\alpha}p_{\alpha}(1,y)}= c_{n,\alpha}
\end{array}\right.
\end{equation*}
for some positive constant $c_{n,\alpha}$, and thus condition \eqref{P3} is satisfied; 
see for instance \cite{Kolo}. We have that 
${p_{\alpha} (t,\cdot)} = {\mathcal F}^{-1}(e^{-t|\xi|^{2\alpha}})$, 
where ${\mathcal F}^{-1}$
denotes inverse Fourier transform.
For $\alpha=1/2$, $p_{1/2}$ admits the explicit expression
$$
p_{1/2}(t,x)=B_n\frac{t}{(t^{2}+\vert x\vert^2)^{(n+1)/2}}
=\frac{B_n}{t^n(1+\vert t^{-1}x\vert^2)^{(n+1)/2}},
$$
where $B_n=\Gamma(\frac{n+1}{2})\pi^{-\frac{n+1}{2}}$ is chosen to ensure property \eqref{P1} above.

More examples of semigroups as above are available in Bony-Courr\`ege-Priouret~\cite{BCP}.
This paper, among many other things, 
characterizes the integral operators satisfying a maximum principle; see Remark~\ref{levyformula}
below.

Our first result concerns a class of initial data in $\R^n$, possibly
discontinuous, which includes compactly supported 
functions. We show that the position of all level sets moves exponentially fast in time. 

\begin{theorem}\label{LIMR}
Let $n\geq 1$, $\alpha\in(0,1)$, $f$ satisfy \eqref{nonl},
and $p$ be a kernel satisfying \eqref{P1}-\eqref{P2}-\eqref{P3}. 

Let $\sigma_\ast=\frac{f'(0)}{n+2\alpha}$.
Let $u$ be a solution of \eqref{N},
where $u_0\not \equiv 0$, $0\leq u_0 \leq 1$ is measurable, and
$$
u_0(x)\leq C|x|^{-n-2\alpha} \ \ \ \text{for all } x\in\mathbb{R}^n
$$
and for some constant $C$.  Then,
\begin{enumerate}
    \item[a)] if $\sigma> \sigma_\ast$, then $u(t,x)\to 0$
uniformly in $\left\{|x|\geq e^{\sigma t}\right\}$ as $t\to +\infty$.
    \item[b)] if $\sigma<\sigma_\ast$, then $u(t,x)\to 1$ 
uniformly in $\left\{|x|\leq e^{\sigma t}\right\}$ as $t\to +\infty$.
\end{enumerate}
\end{theorem}

Part b) on convergence towards $1$ is the delicate part of the theorem.
A simpler result ---and first step towards the previous theorem--- is the following.

\begin{lemma} \label{epslemma}
Under the assumptions of Theorem \ref{LIMR}, for every $\sigma<\sigma_\ast$
there exists $\varepsilon \in (0,1)$ and $\underline{t}>0$ such that
\begin{equation}\label{e3.2}
u(t,x)\geq\varepsilon\ \ \ \text{for all }t\geq \underline{t}\text{ and } |x|\leq e^{\sigma t}.
\end{equation}
\end{lemma}

Even if this lemma concerns initial data decaying at infinity,
from it we can easily deduce the nonexistence of planar traveling waves
(under no assumption of their behavior at infinity, as in the following statement).

\begin{proposition}\label{NTW}
Let $n\geq 1$, $\alpha\in(0,1)$, $f$ satisfy \eqref{nonl},
and $p$ be a kernel satisfying \eqref{P1}-\eqref{P2}-\eqref{P3}. 

Then, there exists no nonconstant planar traveling wave solution of \eqref{N}.
That is, all solutions of \eqref{N} taking values in $[0,1]$ and
of the form $u(t,x)=\varphi (x + t e)$, for some vector $e\in\R^n$, are identically $0$ or $1$.
Equivalently, the only solutions $\varphi : \R^n\to [0,1]$ of
\begin{equation}\label{TW-frac-lapl}
(-\Delta)^\alpha \varphi  + e\cdot \nabla\varphi =f(\varphi)\ \ \ \ \text{in } \mathbb{R}^n
\end{equation}
are $\varphi\equiv 0$ and $\varphi \equiv 1$.
\end{proposition}

The last statement on the elliptic equation \eqref{TW-frac-lapl} has an analogue
for the Laplacian. As shown in \cite{AW}, if $|e|<2\sqrt{f'(0)}$ then equation
\eqref{TW-frac-lapl} with $\alpha=1$ admits the constants $0$ and~$1$ as only solutions
taking values in $[0,1]$.

Our results were already announced in \cite{CR}. 
Also for $\alpha \in (0,1)$, Berestycki, Rossi, and the second author \cite{BRR} have proved 
that there is invasion of the unstable state by the stable one. For a large class of nonlinearities,
Engler~\cite{Eng} has proved that the invasion has unbounded speed. 
Here we prove that for KPP nonlinearities the position
of the front is exponential in time. For another type of integro-differential equations 
Garnier~\cite{Jimmy} also establishes that the position of the level sets move exponentially in time.
Finally, exponentially propagating solutions exist in the standard KPP equations 
as soon as the initial datum decays algebraically; this fact has been noticed by 
Hamel and Roques \cite{HR}.

When $A=-\Delta$, the minimal speed $c_{\ast}$ appears when linearizing around
the leading edge of the front, that is, at $u=0$. In fact, since $f$ is concave,
the solution $\overline{u}$ of
$$
\overline{u}_t-\Delta \overline{u}=f'(0)\overline{u}\quad \text{ and }  
\quad \overline{u}(0,\cdot)=u(0,\cdot)\quad\text{ in }\R^n
$$
is a supersolution of \eqref{N}. Looking at the particular case 
$\overline{u}(0,\cdot )=\delta_0$, the Dirac mass at 0, we obtain  
$\overline{u}(t,x)=(4\pi t)^{-\frac{n}2}e^{f'(0)t-\frac{\vert x\vert2}{4t}}$.
Thus, $\overline{u}=\lambda$ if $\vert x\vert =2\sqrt{f'(0)}t+\textrm{o}(t)$. 

Let us make the same heuristic argument ---already done for instance in \cite{MVV,Cas,Cas3}---
when  $0<\alpha<1$ and \eqref{P3} holds. Now the solution $\overline{u}$ of
$$
\overline{u}_t+A\overline{u}=f'(0)\overline{u}\quad \text{ and }  
\quad  \overline{u}(0,\cdot )=\delta_0\quad\text{ in }\R^n
$$
is
$$
\overline{u}(t,x)= e^{f'(0)t} p(t,x). 
$$
Estimate \eqref{P3} gives that $\overline{u}=\lambda$ if
$\vert t^{-\frac{1}{2\alpha}}x\vert^{n+2\alpha} =t^{-\frac{n}{2\alpha}}e^{f'(0)t}\,\textrm{O}(1)$,
that is, if
\begin{equation}
\label{e3.1}
\vert x\vert =  t^{\frac{1}{n+2\alpha}}\, e^{\sigma_\ast t}\, \textrm{O}(1), 
\quad\text{where }\sigma_\ast=\di\frac{f'(0)}{n+2\alpha} 
\end{equation}
is the same exponent as in Theorem~\ref{LIMR}. 
However, our two next results show that linearizing at the front edge is not 
as accurate in the presence of fractional diffusion as it is for Brownian diffusion.

First, we will see that the exponent $\sigma_\ast$ in \eqref{e3.1} is not the right one
for nondecreasing initial data in $\R$. The front will propagate faster, in fact with an exponent 
larger than~$\sigma_\ast$. Thus, the mass located far away
from the edge of the front (that is, the mass at $+\infty$ present in a nondecreasing solution) 
does play a role in the front speed. This is due to the jumps in the 
underlying L\'evy process.

Second, even that $\sigma_\ast$ is the precise
exponent for compactly supported data, the factor $t^{\frac{1}{n+2\alpha}}$
in \eqref{e3.1} is not correct. It does not appear in the correct expression for the
position of the front, at least for $n=1$ and $A=(-\Delta)^{1/2}$; see Theorem~\ref{LEVR}.
Contrary to the situation in the previous paragraph, here the front travels slower
than the linear leading edge prediction. This is not a surprise: it is typical of the behaviour of Fisher-KPP type fronts. 
In the case $\alpha=1$ with, say, $n=1$ and $f(u)=u-u^2$, even that the leading edge analysis predicts the correct location 
of the front (if $s(t)$ is the first point where $u$ takes the value $1/2$, then $s(t)\sim 2t$ as $t\to+\infty$, 
as can easily be computed from the Gaussian kernel), a purely linear analysis would
predict $s(t)= 2t-\frac{1}{2}{\mathrm {ln}}t+O(1)$, whereas the correct expansion is 
$s(t)= 2t-\frac{3}{2}{\mathrm {ln}}t+O(1)$ (Bramson \cite{Br}).

In one space dimension, it is of interest to understand the dynamics of nondecreasing initial
data. As mentioned before, for the standard Laplacian
the level sets of $u$ travel with the speed $c_\ast$, provided that $u(0,\cdot)$ decays sufficiently 
fast at $-\infty$. In the fractional case, the mass at $+\infty$ has an effect and what happens is 
not a mere copy of the result of Theorem \ref{LIMR} for compactly supported data.
The mass at $+\infty$ makes the front travel faster to the left, indeed with a larger exponent
than $\sigma_\ast$.
In the following theorem, we may take the initial datum to be for instance  
$u_0(x)=H(x)$, the Heaviside function, or even 
$u_0(x)=lH(x)$ for any constant $l\in(0,1]$.

\begin{theorem}\label{LIMI}
Let $n= 1$, $\alpha\in(0,1)$, $f$ satisfy \eqref{nonl},
and $p$ be a kernel satisfying \eqref{P1}-\eqref{P2}-\eqref{P3}.

Let $\sigma_{\ast\ast}=\frac{f'(0)}{2\alpha}$.
Let $u$ be a solution of \eqref{N}, 
where $0\leq u_0 \leq 1$ is measurable and nondecreasing, $u_0 \not\equiv 0$, and 
$$
u_0(x)\leq C(-x)^{-2\alpha} \ \ \ \text{if } x<0
$$
for some constant $C$. Then,
\begin{enumerate}
    \item[a)] if $\sigma >\sigma_{\ast\ast}$, $u(t,x)\to 0$ 
uniformly in $\left\{x\leq-e^{\sigma t}\right\}$ as $t\to +\infty$.
    \item[b)] if $\sigma <\sigma_{\ast\ast}$, $u(t,x)\to 1$ 
uniformly in $\left\{x\geq-e^{\sigma t}\right\}$ as $t\to +\infty$.
\end{enumerate}
\end{theorem}

Note that 
$$
\sigma_{\ast\ast}=\frac{f'(0)}{2\alpha}>\frac{f'(0)}{1+2\alpha}=\sigma_{\ast},
$$
where $\sigma_{\ast}$ is the exponent in Theorem~\ref{LIMR} for $n=1$ and
compactly supported data. Notice also the slower
power decay assumed in the initial condition with respect to Theorem~\ref{LIMR}.
One could wonder whether a model with such features is physically, or biologically relevant. In fact,
this behaviour is consubstantial to fast diffusion, and the model may be relevant to explain 
fast recolonisation events in ecology; see a discussion in \cite{HR}.

Our final result concerns the case $n=1$, $A=(-\Delta)^{1/2}$, $f(u)=u(1-u)$, and 
initial data decaying fast enough at $\pm\infty$. It shows that the factor 
$t^{\frac{1}{n+2\alpha}}=t^{\frac{1}{2}}$ in \eqref{e3.1} does not appear in the front
position.

\begin{theorem}\label{LEVR}
Let $n=1$, $A=(-\Delta)^{1/2}$, and $f(u)=u(1-u)$. That is, we consider the problem 
\begin{equation} \label{N12}
\left \{ \begin{array}{rcll}
u_t +(-\Delta)^{1/2}  u & = & u(1-u) &\ \ \ \ \text{in} \ 
(0,+\infty)\times \mathbb{R} , \\
u(0,\cdot)& = & u_0  &\ \ \ \ \text{in} \ \mathbb{R}.
\end{array}\right.
\end{equation}
Let $u$ be a solution of \eqref{N12},
where $0\leq u_0 \leq 1$ is measurable, $u_0\not \equiv 0$, and
$$
u_0(x)\leq C|x|^{-2}= C|x|^{-n-2\alpha} \ \ \ \text{for all } x\in\mathbb{R}
$$
and for some constant $C$. 

Then, for all $\lambda\in(0,1)$ there exists a constant $C_\lambda>1$ and a time $t_\lambda$
$($both depending only on $u_0$ and~$\lambda)$
such that, for all $t>t_\lambda$,
\begin{equation}\label{incl}
\{|x|>C_\lambda e^{t/2} \}\subset \{u<\lambda\}\quad\text{ and }\quad
\{|x|<\frac{1}{C_\lambda} e^{t/2} \}\subset \{u>\lambda\}.
\end{equation}
As a consequence, if $t>t_\lambda$ then
$$
\{u=\lambda\}\subset (-C_\lambda e^{t/2},-\frac1{C_\lambda}e^{t/2})\cup 
(\frac1{C_\lambda}e^{t/2},C_\lambda e^{t/2})
$$
and $\{u=\lambda\}$ intersects both intervals.
\end{theorem}

\begin{remark}
{\rm
Jones' symmetrization result \cite{Jon} for the Laplacian also applies to equation
\eqref{eq-frac-lapl} for all $\alpha\in (0,1)$. 
Its statement in the present situation is the following. 
Let $u$ be a solution of \eqref{eq-frac-lapl} such that
$u(0,\cdot)\not\equiv 0$ has compact support
in $\R^n$ and satisfies $0\leq u(0,\cdot)\leq 1$. Let $\lambda\in (0,1)$, $t > 0$, and 
$x_0\in \R^n$ be such that $u(t, x_0) = \lambda$ and $\nabla_x u(t, x_0) \not= 0$.
Then, the normal line to the level set $\{x\in\R^n\, :\, u(t, x) = \lambda \}$ 
through the point $x_0$ intersects the convex hull of the support of
the initial datum $u(0,\cdot)$.

Thus, the level sets of solutions with compactly supported initial 
data look more and more spherical as $t$ increases.
Jones' beautiful proof combines the maximum principle and Hopf's lemma 
with reflections along hyperplanes. All these tools are also available for the
fractional Laplacian.
}
\end{remark}

Let us briefly discuss the main ideas in the proofs of our results.
The supersolutions obtained by solving $\overline{u}_t+A\overline{u}=f'(0)\overline{u}$
give an upper bound for the position (in absolute value) of the level sets.
This leads immediately to parts a) of Theorems~\ref{LIMR} and \ref{LIMI}.  

Part b) on convergence towards $1$ is the delicate point and it is done in two steps.
The first one is the content of Lemma~\ref{epslemma} above.
Its lower bound \eqref{e3.2}
is accomplished by constructing solutions of the equation 
$$
\underline{v}_t+A\underline{v}=\frac{f(\delta)}{\delta}\underline{v}
$$ 
which take values in $(0,\delta)$ ---and, as a consequence of the concavity of $f$, are subsolutions
of \eqref{N}. This is done truncating an initial datum $v_0$ at a level $\varepsilon$, 
where $\varepsilon < \delta$, i.e., considering $\min{(v_0,\varepsilon)}$.
We then solve the linear equation above for $\underline{v}$ with this new datum,
up to the time $T$ where $\underline{v}$ takes the value $\delta$. At this point we compute
how the level sets have propagated.
We then truncate $\underline{v}(T,\cdot)$ at the level~$\varepsilon$ as before,
and we iterate this procedure. 

The convergence towards $1$ is shown using \eqref{e3.2} and a subsolution taking 
values in $[\varepsilon,1]$
built through the linear equation  
$$
\underline{w}_t+A\underline{w}=\frac{f(\varepsilon')}{1-\varepsilon'}(1-\underline{w})
$$ 
for some $0<\varepsilon'<\varepsilon$ and an initial condition involving 
$\vert x\vert^\gamma$, with $\gamma\in (0,2\alpha)$.
Here again we use the concavity of $f$ to ensure that
$f(\varepsilon')(1-\varepsilon')^{-1}(1-\underline{w})\leq f(\underline{w})$
for $\underline{w}\in [\varepsilon,1]$.

The proof of Theorem~\ref{LEVR} on the level sets of solutions
in the case $n=1$ and $A=(-\Delta)^{1/2}$ uses some of the previous results and,
in addition, more precise sub and supersolutions of the form 
$$
v(t,x)= a \left( 1+\frac{\vert x \vert^2}{b(t)^2}\right)^{-1},
$$
for certain constants $a$ and functions of time $b=b(t)$.

The plan of the paper is the following. In section~2 we prove several results on the
semigroup $T_t$, especially several maximum and comparison principles, as well as
some upper and lower bounds on the flow. Section~3 is devoted to prove
Proposition~\ref{NTW} on traveling waves 
and Theorem~\ref{LIMR} on solutions with $0$ limit at infinity.
Section~4 concerns Theorem~\ref{LIMI} on increasing solutions in $\R$.
Finally, section~5 establishes Theorem~\ref{LEVR} on precise bounds for the level sets.

\section{The semigroup and its generator: maximum principles and bounds}

In this section we prove several results regarding the semigroup
\begin{equation} \label{solHom}
T_tu_0(x):=
\int_{\mathbb{R}^n}p
(t,y)u_0(x-y)\, dy
=\int_{\mathbb{R}^n}p
(t,x-y)u_0(y)\, dy
\end{equation}
for $u_0\in L^\infty(\R^n)$. We refer to \cite{CH,Pa,Ta} as good monographs in the subject;
the last one puts especial emphasis on Feller semigroups.
Through the paper, all what we assume is that the 
continuous function 
$p$ satisfies \eqref{P1}-\eqref{P2}-\eqref{P3}.

Let us mention here an important situation in which such functions or kernels $p$ arise.
Let $(\{X_t\}_{t\geq 0}, P^x)$ be a Markov process on $\R^n$ with transition probability function 
$P_t(x,dy)$. The quantity $\int_E P_t(x,dy)$ is the probability that a particle, initially at 
$x$, belongs to a Borel set $E$ at time $t$. 
If $P_t(x,dy)$  has a density $p(t,x,y)$ and the process is invariant under translations,
$p(t,x,y)=p(t,x-y)$, then the semigroup property for \eqref{solHom} is just the
conditioned probabilities formula.
This is the framework when $p$ is the classical heat or Gaussian kernel
(the Markov process is then the Wiener or Brownian process), and also
when $p=p_\alpha$ is the kernel for the fractional Laplacian $(-\Delta)^\alpha$
(we then have the symmetric $2\alpha$-stable L\'evy process).

\subsection{The semigroup in $\Con$, in $\Cubn$, and in $X_\gamma$}

Even that the semigroup is well posed in  $L^\infty(\R^n)$, for some proofs it will be important 
to have it defined in some Banach spaces of functions where the
semigroup is strongly continuous. We recall that
{\it a strongly continuous semigroup} in a Banach space $X$ is a family $\{T_t\}_{t>0}$ 
of bounded linear operators on $X$ such that
$T_{t+s} = T_t  T_s$ for all positive $s$ and $t$, and such that
\begin{equation*}
\lim_{t\downarrow 0} \|T_t u - u\|=0 \qquad\text{ for every } u\in X,
\end{equation*}
where $\|\cdot\|$ is the norm in $X$.

Given the definition \eqref{solHom} of our semigroup, the last condition concerns the quantity
\begin{equation*}
(T_tu-u)(x)=\int_{\mathbb{R}^n}p (t,y) \left\{u(x-y)-u(x)\right\} \, dy.
\end{equation*}
Using \eqref{P3} and making the change of variables $t^{-\frac{1}{2\alpha}}y=\overline{y}$,
we obtain
\begin{equation}\label{strongcont2}
|(T_tu-u)(x)|\le \int_{\mathbb{R}^n} \frac{B}{1+|\overline{y}|^{n+2\alpha}}
| u(x-t^{\frac{1}{2\alpha}}\overline{y})-u(x)|\, d\overline{y}.
\end{equation}
We write it as the sum of two
integrals, one in a sufficiently large ball and the other in its complement.  
In this way we see that, in order to have this quantity tend to $0$ as $t\to 0$ uniformly in $x\in\R^n$, 
it suffices for $u$ to be bounded and uniformly continuous in $\R^n$.

Therefore we will work in the spaces
$$
\Cubn =\{ u:\R^n\to \R\ :\ u \text{ is bounded and uniformly continuous in }\R^n\}
$$
and
$$
\Con =\{u \text{ is continuous in }\R^n \text{ and } 
u(x)\to 0 \text{ as }|x|\to\infty \} \subset \Cubn.
$$
Note that, for $u\in\Con$, the continuity of $u$ and its $0$ limit at $\infty$ guarantee the
boundedness and the uniform continuity of $u$. Both are Banach spaces with the $L^\infty(\R^n)$
(or uniform convergence) norm.

Next we will define a family of Banach spaces $X_\gamma$, with $0\leq\gamma<2\alpha$, for which
$\Cubn = X_0$. Later we will check that $T_t$ maps $X_\gamma$ into itself.
In particular, we will have that
$$
T_t \, \Cubn\subset\Cubn .
$$
Using in addition that
\begin{equation*}
|T_tu(x)|\le \int_{\mathbb{R}^n} \frac{B}{1+|\overline{y}|^{n+2\alpha}}
| u(x-t^{\frac{1}{2\alpha}}\overline{y})|\, d\overline{y},
\end{equation*}
it is easy to verify the $0$ limit at infinity  for $T_tu$ whenever $u\in\Con$. That is: 
$$
T_t\, \Con\subset\Con
$$
for all $t>0$.

Moreover, since both $\Cubn$ and $\Con$ carry the $L^\infty$ norm, 
$T_t$ is a contraction in both spaces, i.e., $\|T_t\|\le 1$.
Note also that 
$$
T_t 1 = 1.
$$ 
Finally (and this is the property for a semigroup to be called a Feller semigroup), we have:
$$
\text{if }0\leq u\leq 1, \quad\text{then }  0\leq T_t u\leq 1.
$$

We will need to use some unbounded comparison functions. Thus, we have to set up the
semigroup (and make it to be strongly continuous) in a larger Banach space containing 
unbounded functions. 
For $0\leq \gamma<2\alpha$, we consider functions $u:\R^n\to\R$ such that
\begin{equation}\label{bddgamma}
|u(x)|\le C(1+|x|^\gamma) \ \text{ in }\R^n\text{ for some constant }C 
\end{equation}
and such that 
\begin{equation}\label{ucgamma}
\left \{ \begin{array}{l}
\text{for every }\varepsilon >0 \text{ there exists } \delta>0 \text{ such that:}
\vspace{.5mm}
\\
\text{if }x\in\R^n \text { and } |z|\le \delta, 
\text{ then } \displaystyle\frac{|u(x+z)-u(x)|}{1+|x|^\gamma}\leq\varepsilon.
\end{array}\right.
\end{equation}
We define
$$
X_\gamma:=\{  u:\R^n\to \R\ :\ u \text{ satisfies }\eqref{bddgamma} \text{ and }\eqref{ucgamma}\}
$$
endowed with the norm
$$
\|u\|_{X_\gamma}:= \sup_{x\in\R^n} \frac{|u(x)|}{1+|x|^\gamma}.
$$
Note that $X_0=\Cubn$. With this norm, we clearly have the continuous inclusions
$$
\Con\subset\Cubn\subset X_\gamma.
$$
In addition, $X_\gamma\subset C(\R^n)$ ---the space of continuous, possibly unbounded, 
functions in $\R^n$. 

We will also use that the functions
\begin{equation*}
w_\gamma \in X_\gamma , \quad\text{ where } w_\gamma(x)=|x|^\gamma \text{ for } x\in\R^n ,
\end{equation*}
and
\begin{equation*}
W_\gamma \in X_\gamma , \quad\text{ where } W_\gamma(x)=(x_-)^\gamma \text{ for } x\in\R 
\end{equation*}
and $x_-=\max(-x,0)$ is the negative part of $x$.
For this, simply use the inequalities $|x+z|^\gamma-|x|^\gamma\leq |z|^\gamma$
if $\gamma \leq 1$ and 
$|x+z|^\gamma-|x|^\gamma\leq \gamma |x+z|^{\gamma -1}|z|$
if $\gamma >1$.

We need to verify that $X_\gamma$ is a Banach space. Let $\{u_k\}$ be a Cauchy
sequence in $X_\gamma$. It has a pointwise limit $u$ which clearly satisfies \eqref{bddgamma}. 
Now, given $\varepsilon >0$, to control the quantity in \eqref{ucgamma},
we add and subtract the term $(u_k(x+z)-u_k(x))/(1+|x|^\gamma)$.
Since $|u_k(x)-u(x)|/(1+|x|^\gamma)\leq \varepsilon$ for $k$ large enough, it remains to control
the term
$$
\frac{|u(x+z)-u_k(x+z)|}{1+|x|^\gamma} = \frac{|u(x+z)-u_k(x+z)|}{1+|x+z|^\gamma}\,
\frac{1+|x+z|^\gamma}{1+|x|^\gamma}.
$$
Now, we simply use that if $|z|\le 1$ then 
$1+|x+z|^\gamma\leq 1+2^\gamma(|x|^\gamma+|z|^\gamma)\leq (1+2^\gamma)(1+|x|^\gamma)$.

Next, we verify that 
\begin{equation}\label{boundgamma}
T_t:X_\gamma\to X_\gamma \ \text{ is a bounded linear map and }\
\|T_t\|_{X_\gamma}\leq C_\gamma (1+ t^{\frac{\gamma}{2\alpha}})
\end{equation}
for some constant $C_\gamma$ independent of $t$. 
Indeed, making the change of variables as in \eqref{strongcont2},
we have
\begin{equation}\label{quasixgamma}
\frac{|T_t u(x)|}{1+|x|^\gamma} \le \int_{\mathbb{R}^n} \frac{B}{1+|\overline{y}|^{n+2\alpha}}\
\frac{|u(x-t^{\frac{1}{2\alpha}}\overline{y})|}{1+| x-t^{\frac{1}{2\alpha}}\overline{y}|^\gamma}\
\frac{1+| x-t^{\frac{1}{2\alpha}}\overline{y}|^\gamma}{1+|x|^\gamma}\,d\overline{y}.
\end{equation}
The last factor $1+| x-t^{\frac{1}{2\alpha}}\overline{y}|^\gamma\leq 1+2^\gamma(|x|^\gamma + 
t^{\frac{\gamma}{2\alpha}}|\overline{y}|^\gamma)$, and note that the function
$\overline{y}\mapsto (1+|\overline{y}|^\gamma)/(1+|\overline{y}|^{n+2\alpha})$ is integrable.
Thus, $T_tu$ satisfies \eqref{bddgamma}. To verify \eqref{ucgamma} for $T_tu$, we write
\begin{equation*}
\frac{|T_t u(x+z)-T_tu(x)|}{1+|x|^\gamma} \le \int_{\mathbb{R}^n} \frac{B}{1+|\overline{y}|^
{n+2\alpha}}\
\frac{|u(x+z-t^{\frac{1}{2\alpha}}\overline{y}) - u(x-t^{\frac{1}{2\alpha}}\overline{y})|}
{1+|x|^\gamma}\,d\overline{y},
\end{equation*}
and we conclude as before multiplying and dividing by 
$1+|x-t^{\frac{1}{2\alpha}}\overline{y}|^\gamma$. Thus, we have proved assertion \eqref{boundgamma}.

Finally, we check that $\{T_t\}_{t>0}$ is a strongly continuous semigroup in $X_\gamma$.
We have
\begin{equation*}
\frac{|T_t u(x)-u(x)|}{1+|x|^\gamma} \le \int_{\mathbb{R}^n} \frac{B}{1+|\overline{y}|^{n+2\alpha}}\
\frac{|u(x-t^{\frac{1}{2\alpha}}\overline{y}) - u(x)|}
{1+|x|^\gamma}\,d\overline{y}.
\end{equation*}
Given $\varepsilon>0$, the numerator in the second factor in the integral
is controlled by $C(1+|x|^\gamma+|x-t^{\frac{1}{2\alpha}}\overline{y}|^\gamma)$.
Thus,
when integrating in $\{|\overline{y}|>A\}$ the result is smaller than $\varepsilon$
if we take $A$ large enough ---since $\overline{y}\mapsto 
(1+|\overline{y}|^\gamma)/(1+|\overline{y}|^{n+2\alpha})$ is integrable.
Finally, for the integral in $\{|\overline{y}|\leq A\}$, we take $t$ small enough to ensure
$t^{\frac{1}{2\alpha}}\overline{y}\leq\delta$ (where $\delta$ is as in \eqref{ucgamma}),
and the integral becomes smaller than a constant times $\varepsilon$.

\begin{remark}
{\rm
Note that, for $0<\gamma<2\alpha$, 
$T_t:X_\gamma\to X_\gamma$ is a bounded linear map (see \eqref{boundgamma}),
but not necessarily a contraction. Instead, we have that $T_t:\Cubn\to\Cubn$ and
$T_t:\Con\to\Con$ are contractions for all $t>0$.

Recall also that $T_t:L^\infty(\R^n)\to L^\infty(\R^n)$ also form a semigroup of contractions, 
but not a strongly continuous semigroup.
}
\end{remark}

\subsection{The generator of the semigroup}

Given a strongly continuous semigroup in a Banach space $X$, one can define its (infinitesimal) 
generator $-A$  by  
\begin{equation} \label{IG}
-Au= \displaystyle\lim_{t \downarrow 0}{\frac{T_tu-u}{t}} \qquad\text{for } u\in D(A)\subset X,
\end{equation}
where the domain $D(A)$ of $A$ (or $-A$) is the subspace of $X$ defined by
\begin{equation*}
D(A):=\{u\in X \text{ for which the limit in $X$ as $t\downarrow 0$ in \eqref{IG} exists}\}. 
\end{equation*}

We will denote by
\begin{equation*}
D_0(A)\subset D_{u,b}(A)\subset D_\gamma(A)
\end{equation*}
the domain of the generator $-A$ of $\{T_t\}$ in the Banach spaces
$\Con$, $\Cubn$, and $X_\gamma$, respectively. The inclusions of these three domains
are clear because of previous considerations.

To verify that a certain function belongs to $D(A)$ may not be an easy task.
However, the following is a general fact that will be very useful later; see Lemmas~\ref{appR}
and \ref{appI} below. For every strongly continuous
semigroup $\{T_t\}$ in a Banach space $X$, we have:
\begin{equation}\label{intdomain}
\begin{split} 
& \text{ for all }u\in X\text{ and }0\le a < b,\\
&\int_a^b T_s u\, ds \in D(A) \quad\text{ and }\quad A \int_a^b T_s u\, ds =
T_au-T_b u .
\end{split}
\end{equation}
Before verifying this, note that the function $s\geq 0 \mapsto T_s u\in X$ is continuous
(and therefore locally integrable) thanks to the strong continuity of the semigroup.
Note also that from \eqref{intdomain} we deduce that
$$
D(A) \text{ is a dense subspace of } X,
$$
since $t^{-1}\int_0^t T_s u\, ds \in D(A)$ tends to $u$ in $X$ as $t\downarrow 0$.
Now, to verify \eqref{intdomain}, take $0<t<b-a$ and note that
\begin{equation}\label{Aofaverage}
\begin{split}
\frac{T_t-\textrm{Id}}{t} \int_a^b T_s u\, ds\, &=\, \frac{1}{t}\int_{a+t}^{b+t} T_s u\, ds-
\frac{1}{t}\int_a^b T_s u\, ds \\
& = \,\frac{1}{t}\int_{b}^{b+t} T_s u\, ds-
\frac{1}{t}\int_a^{a+t} T_s u\, ds \longrightarrow T_b u-T_a u
\end{split}
\end{equation}
as $t\downarrow 0$.

Similar kind of arguments establish also the following two general results
(see section~1.2 of \cite{Pa}). First,
the generator $A$ is a closed linear operator. Second, and important for later purposes:
\begin{equation*}
\begin{split}
&\text{if } u \in D(A), \, \text{ then } 
\{t\mapsto T_{t} u\} \in C^{1}([0,\infty);X),\,  T_tu\in D(A)\text{ for all } t\geq 0, \text{ and }\\ 
&\frac{d}{dt} (T_{t} u) + A T_{t} u =0 \, \text{ for all } t\geq 0.
\end{split}
\end{equation*}

Let us now recall two important properties of $A$ which follow easily from its definition
\eqref{IG} and the fact that $T_t$ is the convolution with a probability kernel.

We have the following maximum principle:
\begin{equation}\label{MP}
\text{if } u\in D_\gamma(A) \text{ satisfies }u\leq u(x_0)\text{ in $\R^n$ for some $x_0$ 
in $\R^n$, then } Au(x_0)\geq 0.
\end{equation}
Recall that here $\gamma<2\alpha$, that $D_\gamma(A)\subset X_\gamma$ is the domain of $A$
in $X_\gamma$, and that functions in $X_\gamma$ are continuous but may be unbounded.
Thus, we are assuming that this particular $u\in X_\gamma$ is bounded above and achieves its
maximum. Statement \eqref{MP} follows from 
$(T_tu-u)(x_0)=\int_{\R^n} p(t,y)\{u(x_0-y)-u(x_0)\}\, dy\leq 0$ for all $t>0$.

The operator $A$ annihilates constant functions and it is invariant by translations: 
$$
A1\equiv 0\ \text{ and }\
(Au)(x+x_0)=(Au(\cdot+x_0))(x)\ \text{for all } u\in D_\gamma(A), x_0\in\R^n, x \in  \R^n.
$$

The previous maximum principle (and also an important extension to prove our results,
Proposition~\ref{nonlocMPdomain}) apply to functions in the domain of $A$. This will
be sufficient for our results once we show the existence of ``enough'' initial conditions
belonging to the domain of $A$. We do this in the following lemmas. 
An alternative approach would be that of
subsection~ 2.6, in which we prove the needed maximum principle for ``weak'' or mild solutions.
If one chooses this alternative approach, the previous 
considerations on the initial data being in the domain of $A$ may be avoided.

We can now exhibit initial conditions in the domain of $A$
whose nonlinear flow will stay below the flow of the given arbitrary initial condition.
We start with the case of data in $\Con$.

\begin{lemma}\label{appR}
Let $n\geq 1$, $\alpha\in(0,1)$, 
and $p$ be a kernel satisfying \eqref{P1}-\eqref{P2}-\eqref{P3}.
Let $0\leq u_0 \leq 1$ be measurable in $\R^n$ and $u_0 \not\equiv 0$.

Then, for some constant $c>0$ depending on $u_0$,
$$
T_2u_0 \geq c \int_1^2 p(s,\cdot)\, ds \quad\text{and} \quad  \int_1^2 p(s,\cdot)\, ds \in D_0(A).
$$
In addition, $\int_1^2 p(s,x)\, ds \leq C|x|^{-n-2\alpha}$ for all $x\in\R^n$,
for some constant $C>0$.
\end{lemma}

Regarding the semigroup in $C_{u,b}(\mathbb{R})$, we have:

\begin{lemma}\label{appI}
Let $n= 1$, $\alpha\in(0,1)$, 
and $p$ be a kernel satisfying \eqref{P1}-\eqref{P2}-\eqref{P3}.
Let $0\leq u_0 \leq 1$ be measurable and nondecreasing, with $u_0 \not\equiv 0$.
Let
$$
P(t,x):=\int_{-\infty}^x p(t,y)\, dy.
$$

Then, for some constant $c>0$ depending on $u_0$, 
$$
T_2u_0 \geq c \int_1^2 P(s,\cdot)\, ds \quad\text{and} \quad\int_1^2 P(s,\cdot)\, ds \in D_{u,b}(A).
$$
In addition, $\int_1^2 P(s,x)\, ds \leq C(-x)^{-2\alpha}$ for all $x<0$,
for some constant $C>0$.
\end{lemma}

Therefore, given initial data satisfying the hypothese of Theorems~\ref{LIMR} or \ref{LIMI},
we have built smaller initial data (after time $2$) satisfying the same hypothese
of the theorems and belonging to the domain of the semigroup. They will be useful to give
pointwise sense to $Au(t)(x)$ after running the nonlinear problem and, hence,
useful to apply an easy maximum principle proved in subsection~2.5 below.

Let us denote 
\begin{equation}\label{nameq}
q(t,x):= \frac{1}{t^{\frac{n}{2\alpha}}(1+|t^{-\frac{1}{2\alpha}}x|^{n+2\alpha})}
=\frac{t}{t^{\frac{n}{2\alpha}+1}+
|x|^{n+2\alpha}}.
\end{equation}

\begin{proof}[Proof of Lemma~\ref{appR}]
We claim that, for some constant $c>0$ depending only on $n$ and $\alpha$,
\begin{equation}\label{lowercharR}
\begin{split} 
T_t\chi_{B_1(0)}(x) & \geq  B^{-1} (q(t,\cdot)\ast\chi_{B_1(0)})(x)\\
& \geq  c \frac{1}{t^{\frac{n}{2\alpha}}(1+|t^{-\frac{1}{2\alpha}}x|^{n+2\alpha})}
=c \, q(t,x)
\quad\text{ for }|x|\geq 1,\, t>0,
\end{split}
\end{equation}
where $\chi_{B_1(0)}$ denotes the indicator function of the unit ball.

Indeed, by \eqref{P3} we have
$$
T_t\chi_{B_1(0)}(x)
\geq B^{-1}\int_{B_1(0)}\frac{t^{-\frac{n}{2\alpha}}}
{1+(t^{-\frac{1}{2\alpha}} |x-y|)^{n+2\alpha}}\, dy.
$$
In the integral, $|y|\leq 1 \leq|x|$ ---we are taking $|x|\geq 1$ by hypothesis. Thus, 
$|x-y|\leq|x|+|y|\leq 
2|x|$ and the  integrand is larger or equal than 
$t^{-\frac{n}{2\alpha}}(1+(t^{-\frac{1}{2\alpha}} 2|x|)^{n+2\alpha})^{-1}$,
which proves \eqref{lowercharR}. 

Now, notice that $T_1u_0=p(1,\cdot)\ast u_0$ is a positive continuous function.
Hence, $T_1u_0\geq c\chi_{B_1(0)}$ for some positive constant $c$. Thus,
$T_2u_0=T_1T_1u_0\geq T_1 c\chi_{B_1(0)}$. This fact, the lower bound \eqref{lowercharR} 
with $t=1$, and the standing upper bound \eqref{P3} for $p$ in terms of $q$
lead to the lower bound for $T_2u_0$ of the lemma.

The statement that  $\int_1^2p(s,\cdot)\, ds$ belongs to $D_0(A)$ is a particular case of
the general fact \eqref{intdomain} for strongly continuous semigroups. Note here that
$$
\int_1^2 p(s,\cdot)\, ds = \int_0^1 T_\tau \, p(1,\cdot)  \, d\tau.
$$
Finally, the upper bound for $\int_1^2p(s,\cdot)\, ds$ follows from  \eqref{P3}.
\end{proof}

\begin{proof}[Proof of Lemma~\ref{appI}]
We claim that, for some constant $c>0$ depending only on $\alpha$,
\begin{equation}\label{lowercharI}
\begin{split}
T_t\chi_{(0,+\infty)}(x) & \geq   B^{-1} (q(t,\cdot)\ast\chi_{(0,+\infty)})(x) \\
&\geq   B^{-1} c\, (1+|{t^{-\frac{1}{2\alpha}}}x|
)^{-2\alpha} \quad\text{for }x<0, \, t>0.
\end{split}
\end{equation}

Indeed, simply note that, since $x<0$, 
\begin{eqnarray*} 
T_t\chi_{(0,+\infty)}(x) &\geq & B^{-1}\int^{+\infty}_{0}
\frac{t^{-\frac{1}{2\alpha}}}{1+(t^{-\frac{1}{2\alpha}} |x-y|)^{1+2\alpha}}\, dy \\
&=    & B^{-1}  \int^{+\infty}_{-x}\frac{t^{-\frac{1}{2\alpha}}}{1+
(t^{-\frac{1}{2\alpha}} z)^{1+2\alpha}}\, dz
= B^{-1}  \int^{+\infty}_{-t^{-\frac{1}{2\alpha}}x}\frac{d\overline z}{1+
\overline{z}^{1+2\alpha}} \\
&\geq & 
c\, (1+|{t^{-\frac{1}{2\alpha}}}x|)^{-2\alpha}.
\end{eqnarray*}

The rest of the proof is identical to that of the previous lemma. Just note that
\begin{equation}\label{primitive}
P(t+s,\cdot) = T_t\, P(s,\cdot)
\end{equation}
for all positive $s$ and $t$, and hence
$\int_1^2 P(s,\cdot)\, ds=\int_0^1T_\tau P(1,\cdot)\,d\tau$.
\end{proof}

\begin{remark}\label{solnelliptic}
{\rm 
Let $n\geq 1$, $\alpha\in(0,1)$, and $p$ be a kernel satisfying \eqref{P1}-\eqref{P2}-\eqref{P3}.
Since $\{T_t\}$ is a strongly continuous semigroup of contractions both in $\Con$ and
in $\Cubn$, a general result of semigroup theory (see Proposition~3.4.3 of \cite{CH}) 
guarantees that its infinitesimal generator $-A$ is a $m$-dissipative operator.

In particular, given any $g\in\Con$ (respectively, $g\in\Cubn$) and any $\lambda >0$, 
the elliptic equation
\begin{equation}\label{elliptic}
Au+\lambda u = g \quad\text{ in }\R^n
\end{equation}
admits a unique solution $u\in D_0(A)\subset \Con$ (respectively, $u\in D_{u,b}(A)\subset\Cubn$).
It is given explicitly by the formula
$$
u=\int_0^{+\infty} e^{-\lambda t}T_tg \, dt,
$$
that is, 
$$
u(x)=\int_0^{+\infty} \int_{\R^n} e^{-\lambda t} p(t,x-y)g(y) \, dt\, dy.
$$
It is simple to check that $u$, defined in this way, belongs to the domain of $A$
and satisfies \eqref{elliptic} (see Proposition~3.4.3 of \cite{CH}). 
For the uniqueness statement, note that, for all $s>0$
and all $u\in \Cubn$,
\begin{eqnarray*}
\left\|-s^{-1}(T_su-u)+\lambda u\right\|_\infty & \geq & 
\left\|(\lambda +s^{-1})u\right\|_\infty -
\left\|s^{-1}T_su\right\|_\infty \\
& \geq & \left\|(\lambda +s^{-1})u\right\|_\infty -
\left\|s^{-1}u\right\|_\infty=\lambda \left\|u\right\|_\infty .
\end{eqnarray*}
By letting $s\downarrow 0$, if $u\in D_{u,b}(A)$ then 
$\|Au+\lambda u\|_\infty \geq  \lambda \|u\|_\infty$.
In particular, if $u$ solves \eqref{elliptic} with $g\equiv 0$, then $u\equiv 0$.
}
\end{remark}

\begin{remark}\label{levyformula}
{\rm 
Under the additional assumption that $C_c^\infty(\R^n)$ (i.e., $C^\infty$ functions with
compact support) is contained in 
$D(A)$ (that we do not make in this paper), 
\cite{BCP} (see Theorems IX and XIV) characterized the generators $A$ of Feller
semigroups. Restricted to 
$C_c^\infty(\R^n)$, $A$ is the sum of a local diffusion (second-order) operator
and an integro-differential operator of L\'evy type. See also Theorem~9.4.1 of \cite{Ta}.
In particular, \cite{BCP} characterizes the integral operators satisfying the maximum principle.
}
\end{remark}

\subsection{The nonlinear problem. Comparison principle}

Throughout this section, $A$ is the generator of a strongly continuous semigroup
in a Banach space $X$.

We recall the notion of mild solution for the nonhomogeneous linear problem
\begin{equation} \label{Nh}
\left \{ \begin{array}{rcll}
u_t + A u & = & h(t) &\  \ \text{in } (0,T), \\
u(0)& = & u_0,  &
\end{array}\right.
\end{equation}
where $T>0$, $u_0\in X$, and $h\in C([0,T];X)$ are given. The
mild solution of \eqref{Nh} (see \cite{Pa}) is given explicitly by Duhamel's principle
(or formula of the variation of constants): 
\begin{equation*}
u(t) = T_tu_0 +\int_0^t T_{t-s}\, h(s) \, ds
\end{equation*}
for all $t\in [0,T]$. One easily checks that $u\in C([0,T];X)$.

We now turn to the nonlinear problem. Let $G:[0,\infty)\times X\to X$,
$G=G(t,u)$ be a function satisfying
\begin{equation}\label{hypg}
\begin{split}
&G\in C^1([0,\infty)\times X; X)\ \ \text{ and }\\
&G(t,\cdot)\ \text{ is globally Lipschitz  in $X$
uniformly in $t\geq 0$.}
\end{split}
\end{equation}
Given any $T>0$, we are interested in the nonlinear problem 
\begin{equation} \label{Ng}
\left \{ \begin{array}{rcll}
u_t + A u & = & G(t,u) &\  \ \text{in} \ (0,T), \\
u(0)& = & u_0,  &
\end{array}\right.
\end{equation}
for a given $u_{0}\in X$. In our case (in which $X$ is a subspace of $\Cubn$), 
$G$ will be given by
\begin{equation} \label{Gtu}
G(t,u)(x):=g(t,x,u(x)),
\end{equation} 
where $g:[0,\infty)\times\R^n\times\R\to \R$ is a given nonlinearity.
We say that $u\in C([0,T];X)$ is a
mild solution of \eqref{Ng} if 
\begin{equation} \label{Ngmild}
u(t) = T_tu_0 +\int_0^t T_{t-s}\, G(s, u(s)) \, ds
\end{equation}
for all $t\in [0,T]$. 

Note that the map $N_{u_0}: C([0,T];X)\to C([0,T];X)$ given by
\begin{equation}\label{contraction}
N_{u_0}(u) (t):= T_tu_0 +\int_0^t T_{t-s}\, G(s, u(s)) \, ds
\end{equation}
is Lipschitz in $C([0,T];X)$
with Lipschitz constant 
\begin{equation}\label{mu}
\|N_{u_0}\|_{\text{Lip}}\leq T\, M\, \text{Lip}_{u}(G),
\end{equation}
where $\text{Lip}_{u}(G)$ denotes the Lipschitz constant of $G$ in $u$,
and $M:=\sup_{t\in [0,T]} \|T_{t}\|$. Recall that for any
strongly continuous semigroup, we have that
$\|T_{t}\|\leq Ce^{\omega t}$ for some constants $C$ and $\omega$; see Theorem~2.2 in chapter~1 
of \cite{Pa}.
Using \eqref{mu} (also for the maps $N_{u_0}$ defined in $C([0,\tau];X)$ with $\tau<T$)
and expression \eqref{Ngmild}, it 
follows by induction that
$(N_{u_0})^k$ is Lipschitz in $C([0,T];X)$
with Lipschitz constant $\{T\, M\, \text{Lip}_{u}(G)\}^k/ k!$, where $k$ is any positive integer. 
This constant
is less than~$1$ if we take $k$ large enough. Now, by an easy extension of the contraction principle,
not only $(N_{u_0})^k$ but also $N_{u_0}$ has a unique fixed point. Thus, there exists a unique mild 
solution $u$ of
\eqref{Ng} for every $T>0$. It is also easy to see that it is given by the limit of the iterates 
$(N_{u_0})^{i} (v)$, $i\in \Z^{+}$, 
of any given element $v\in C([0,T];X)$. In particular, taking $v=v(t)\equiv u_0$, we have
\begin{equation}\label{iterates}
u=\lim_{i\to+\infty} (N_{u_0})^{i} (u_{0}).
\end{equation}

Given $0<T<T'$, the mild solution in $(0,T')$ must coincide
in $(0,T)$ with the mild solution in this interval, by uniqueness. Thus, 
under assumption \eqref{hypg}, the mild solution of \eqref{Ng} extends uniquely
to all $t\in [0,\infty)$, i.e., it is global in time.
This applies, in particular, to the linear problem $u_t+Au=au$, with $a\in\R$, in the Banach space~$X_\gamma$.

Next, a useful fact for several future purposes.
We claim that if $u_0\in X$, $u$ is the mild solution of \eqref{Ng}, and $a\in\R$, then
\begin{equation}\label{newu}
\tilde{u}(t):=e^{at}u(t)
\end{equation}
is the mild solution of
\begin{equation} \label{Nga}
\left \{ \begin{array}{rcll}
\tilde{u}_t + A \tilde{u} & = & \tilde G(t,\tilde{u}) &\  \ \text{in} \ (0,T), \\
\tilde{u}(0)& = & u_0,  &
\end{array}\right.
\end{equation}
where
\begin{equation} \label{defga}
\tilde{G}(t,\tilde{u}) := a\tilde{u}+ e^{at}G(t,e^{-at}\tilde{u}).
\end{equation}
Note that $\tilde{G}$ also satisfies \eqref{hypg}, as $G$ does. 

To verify this fact, denote $h(s):=G(s, u(s))$ and use \eqref{Ngmild} with $t$ replaced by~$s$,
i.e., $u(s) = T_su_0 +\int_0^s T_{s-\tau}\, h(\tau) \, d\tau$.
Hence, for $0\leq s\leq t$, $T_{t-s}u(s) = T_tu_0 +\int_0^s T_{t-\tau}\, h(\tau) \, d\tau$.
We now multiply by $ae^{as}$, integrate in $s$, and use that the function
$\int_0^s T_{t-\tau}\, h(\tau) \, d\tau$ is differentiable in $s$ in order to
integrate by parts. We have
\begin{eqnarray}\label{intbyparts}
\int_0^t ae^{as} T_{t-s}u(s)\,ds &=& \int_0^t ae^{as} T_tu_0 \,ds+
\int_0^t ds\,  ae^{as} \int_0^s \, d\tau\, T_{t-\tau}\, h(\tau) \\
&=& (e^{at}-1) T_tu_0 +
e^{at}\int_0^t d\tau\, T_{t-\tau}\, h(\tau)
-\int_0^t ds\, e^{as}T_{t-s}\, h(s)\nonumber\\ \nonumber
&=& e^{at}u(t) - T_tu_0 -\int_0^t ds\, e^{as}T_{t-s}\, h(s).
\end{eqnarray}
This is equivalent to what we needed to show:
$$
\tilde{u}(t):=e^{at}u(t) =T_tu_0 +\int_0^t ds\, e^{as} T_{t-s}\,(au(s)+h(s)).
$$

In particular, if $g(t,x,u)=au$ then $u(t)=e^{at}T_{t}u_{0}$ is the mild solution of 
\eqref{Ng}-\eqref{Gtu}
for all $T>0$. This follows after considering \eqref{newu}-\eqref{Nga}-\eqref{defga}
with $a$ replaced by $-a$, since in this case $\tilde G= 0$.

We now apply all these facts to problem \eqref{N}.
Recall our standing assumption \eqref{nonl} for the nonlinearity $f$.  
Now, we extend $f$ outside $[0,1]$
to ensure that
\begin{equation}\label{hypf}
f\in C^1(\R)\text{ is globally Lipschitz and 
$f'$ is uniformly continuous in $\R$.}
\end{equation}
We work in the Banach spaces
$$
X=\Con \ \text{ and } \ X=\Cubn.
$$
Taking $G(t,u)(x):=f(u(x))$ we can verify \eqref{hypg}. We use that $f'$ is uniformly continuous
to check that the map $u\in\Cubn\mapsto f(u)\in\Cubn$ is continuously differentiable.
We also use $f(0)=0$ to ensure $u\in\Con\mapsto f(u)\in\Con$. Thus, by the previous considerations,
there is a unique mild solution $u$ of 
\begin{equation} \label{NT}
\left \{ \begin{array}{rcll}
u_t + A u & = & f(u) &\ \ \ \ \text{in} \ 
(0,\infty)\times \mathbb{R}^n , \\
u(0,\cdot)& = & u_0  &\ \ \ \ \text{in} \ \mathbb{R}^n  ,
\end{array}\right.
\end{equation}
for data $u_{0}$ in any of both Banach spaces. 

We now claim the following comparison principle. 
Assume that $f_{1}$ and $f_{2}$
satisfy \eqref{hypf} and $f_1\leq f_2$ in $\R$. 
We then have: 
\begin{equation}\label{compar}
\text{if } u_{1}(0,\cdot) \leq u_{2}(0,\cdot)  \text{ belong to } \Cubn, \text{ then }  
u_{1}(t,\cdot)\leq u_{2}(t,\cdot)
\end{equation}
for all $t\in [0,\infty)$, where $u_{1}$ and $u_{2}$ are the respective mild solutions of the nonlinear 
problem \eqref{NT} with $f$ and $u_{0}$ replaced by $f_{i}$ and $u_{i}(0,\cdot)$. 

This is verified as follows. Take $a:= \max\{\text{Lip}(f_{1}),\text{Lip}(f_{2})\}$ to ensure that 
$$
\tilde{g}_{i}(t, \tilde{u}) := a\tilde{u}+ e^{at}f_{i}(e^{-at}\tilde{u})
$$
are nondecreasing in $\tilde{u}$. We know that the mild solution to problem
\eqref{Nga} with $T=\infty$, $\tilde G=\tilde g_{i}$ and initial data
$u_{i}(0,\cdot)$ is given by $\tilde{u}_{i}(t)=e^{at}u_{i}(t)$. 
Hence, \eqref{compar} is equivalent to
\begin{equation*}
\tilde u_{1}(t,\cdot)\leq \tilde u_{2}(t,\cdot)\ \ \text{ for all } t\in [0,\infty).
\end{equation*}
Now, by \eqref{iterates}, it is enough by induction to show that $N_{1}(\tilde{w}_{1})(t)\leq
N_{2}(\tilde{w}_{2})(t)$ for all $t\in [0,\infty)$ whenever $\tilde{w}_{1}(t)\leq\tilde{w}_{2}(t)$
for all $t\in [0,\infty)$. Here $N_{i}$ denotes the map \eqref{contraction} with $g$ replaced
by $\tilde g_{i}$ and $u_{0}$ replaced by $u_{i}(0,\cdot)$. 
This fact is obvious since $u_{1}(0,\cdot)\leq u_{2}(0,\cdot)$, $\tilde g_{i}$ are nondecreasing in 
$\tilde u$, $f_1\leq f_2$, and $T_{t}$ is order preserving.

As a consequence of this comparison principle, the solution $u$ of
\eqref{N} satisfies $0\leq u\leq 1$ in all $[0,+\infty)\times\R^{n}$ for every $u_{0}\in
\Cubn$ with $0\leq u_{0}\leq 1$. We simply use that $u\equiv 0$ and $u\equiv 1$
are solutions of the same problem with smaller and bigger initial data, respectively. 

\begin{remark}\label{regularity}
{\rm
If the initial datum belongs to the domain of $A$, we have further regularity in~$t$
of the mild solution $u=u(t)$. This follows from Theorem~1.5 in section 6.1 of \cite{Pa}
and its proof; see also Definition~2.1 in section 4.2 of \cite{Pa}.
Under hypothesis
\eqref{hypg} (here the continuous differentiability of $G$ with values in $X$ is important),
the mild solution $u$ of \eqref{Ng} satisfies
\begin{equation}\label{diffdomain}
u\in C^1([0,T);X) \ \text{ and }\  u([0,T))\subset D(A)
\quad\text{ if } u_{0}\in D(A),
\end{equation}
and it is a classical solution, i.e., a solution satisfying \eqref{Ng} pointwise 
for all $t\in (0,T)$. In particular, this is the case for the linear problem, $G(t,u)=au$.

As a consequence, if the initial datum $u_{0}$ in \eqref{N}
belongs to the domain $D_{0}(A)$ (respectively, $D_{u,b}(A)$), then the mild solution 
$u$ of \eqref{N} satisfies \eqref{diffdomain} (with $D(A)=D_{0}(A)$, respectively 
$D(A)=D_{u,b}(A)$) and it is a classical solution.
}
\end{remark}

Finally, we need the following proposition describing the solution of \eqref{N} 
corresponding to 
nondecreasing initial conditions in $\R$ with a limit $l\in (0,1]$ at $+\infty$.
To prove it we will use the function
\begin{equation*}
V_1(x):=\int_1^2 ds \int_{-\infty}^x dy\,\, p(s,y).
\end{equation*}
It agrees with the function $\int_1^2P(s,\cdot)\,ds$ considered in Lemma~\ref{appI}.
By that lemma  and by \eqref{primitive} and \eqref{Aofaverage}, we know that
\begin{equation}\label{vone2}
V_1\in D_{u,b}(A) \quad\text{ and }\quad AV_1=P(1,\cdot)-P(2,\cdot)\in C_0(\R).
\end{equation}
In addition, it is clear that
\begin{equation}\label{vone3}
\lim_{x\to -\infty} V_1(x)=0 \quad\text{ and }\quad \lim_{x\to +\infty} V_1(x)=1.
\end{equation}

\begin{proposition}\label{limitsinfincr}
Let $n=1$, $\alpha\in(0,1)$, $f$ satisfy \eqref{nonl},
and $p$ be a kernel satisfying \eqref{P1}-\eqref{P2}-\eqref{P3}.
Let $u_0\in D_{u,b}(A)$ satisfy $0\leq u_0 \leq 1$,
$$
\lim_{x\to -\infty} u_0 (x)=0\quad\text{and}\quad \lim_{x\to +\infty} u_0(x)=l,
$$
where $0<l\leq 1$ is a constant. Let $u$ be the mild solution of \eqref{N}. 
Let $\phi_l=\phi_l(t)$ be the solution of
\begin{equation*}
\phi_l'=f(\phi_l) \ \ \text{ in } [0,\infty), \quad \phi_l(0)=l.
\end{equation*}

Then, the function
$$
v(t,x):=u(t,x)-\phi_l(t)V_1(x) \quad\text{satisfies } v\in C^1([0,\infty);\Co). 
$$
In particular, $\lim_{x\to -\infty} u (t,x)=0$ and $\lim_{x\to +\infty} u(t,x)=\phi_l(t)$,
both uniformly in $t\in [0,T]$, for every $T$.
\end{proposition}

Note that the limits at $\pm\infty$ claimed for the solution are consequence of
the statement  $v(t)\in \Co$. In addition, since $f(\phi)\simeq |f'(1)| (1-\phi)$ 
near $\phi = 1$, we have that
$\phi_l(t) \simeq 1- ce^{-|f'(1)|t}$ for $t$ large, with $c$ a positive constant.

\begin{proof}[Proof of Proposition \ref{limitsinfincr}]
Consider $v=v(t,x)$ as in the statement of the proposition. 
Since we assume $u_0\in D_{u,b}(A)$, by Remark \ref{regularity}
the solution $u$ is classical. Since in addition $V_1\in D_{u,b}(A)$ by \eqref{vone2}, we have
\begin{eqnarray*}
(v_t+Av)(t,x)&=& f(u(t,x))-f(\phi_l(t))V_1(x)-\phi_l(t) AV_1(x)\\
&=& f\left( v(t,x)+\phi_l(t)V_1(x)\right) -f(\phi_l(t))V_1(x)-\phi_l(t) AV_1(x).
\end{eqnarray*}
Therefore, $v$ solves 
\begin{equation}\label{problemv}
\left \{ \begin{array}{rcll}
v_t + A v & = & g(t,x,v) &\  \ \text{in } (0,\infty), \\
v(0)& = & u_0-lV_1,  &
\end{array}\right.
\end{equation}
where
$$
g(t,x,v):= f\left( v+\phi_l(t)V_1(x)\right) -f(\phi_l(t))V_1(x)-\phi_l(t) AV_1(x)
$$
for $t\in [0,\infty)$, $x\in \R$, and $v\in\R$. 

Let $X=\Co$ and $G$ defined by $G(t,v)(x)=g(t,x,v(t,x))$. 
Using that $f(0)=0$, \eqref{vone2}, and \eqref{vone3}, one checks that 
$G:[0,\infty)\times \Co\to \Co$ and that $G$ satisfies \eqref{hypg} with Lipschitz
constant $\text{Lip}(f)$. Thus, by previous considerations in this subsection,
$v$ is the unique classical solution of \eqref{problemv} in $X=\Co$.
In particular, $v\in C^1([0,\infty);\Co)$. {From} this, the last statement of the 
proposition follows easily.
\end{proof}

\subsection{The nonlinear problem for discontinuous initial data. Comparison principle}

Even that our semigroup is not strongly continuous in $L^\infty(\R^n)$, here we 
show that, for initial datum $u_0\in L^\infty(\R^n)$,  
our nonlinear problem \eqref{N} admits a unique mild solution which is
global in time. In addition, the comparison principle of the last subsection still
holds for bounded (perhaps discontinuous) initial data.

One starts writing the notion of mild solution of \eqref{N}:
\begin{eqnarray*}
u(t,x) &=& (T_tu_0) (x) +\int_0^t T_{t-s}\, f(u(s,x)) \, ds\\
&=& \int_{\R^n} dy\  p(t,x-y)u_0(y) + \int_0^t ds \int_{\R^n} dy\ p(t-s,x-y) f(u(s,y)),
\end{eqnarray*}
for $t\in (0,T)$ and $x\in\R^n$, where $u_0\in L^\infty(\R^n)$ is given. 
Since the map given by the right hand side is not continuous 
in time with values in $L^\infty(\R^n)$, we now work in 
the Banach space $L^\infty((0,T)\times\R^n)$. The map is clearly Lipschitz in 
$L^\infty((0,T)\times\R^n)$ with Lipschitz constant $T\text{Lip}(f)$.
By the same variant of the contraction principle  used in the previous subsection,
we conclude the existence and uniqueness of a global in time mild solution of \eqref{N} with 
$$
u\in L^\infty((0,T)\times\R^n)\quad\text{ for all } T>0.
$$

To prove the comparison principle ---as stated in the previous subsection---
we proceed in the same way as there. The only point to check is the statement about the mild
solution for the new function \eqref{newu} and nonlinearity \eqref{defga}.
The argument is the same as there since we can integrate by parts in \eqref{intbyparts}
due to the absolute continuity of 
$\int_0^s \, d\tau\, T_{t-\tau}\, h(\tau)$ in $s$, which allows to use the 
fundamental theorem of calculus.

As a consequence of this comparison principle, 
if $u$ is the mild solution of \eqref{N} with  
$n=1$ and $u_0\in [0,1]$ measurable and nondecreasing (recall that this means $u_0(\cdot+x_0)-u_0\geq 0$ 
a.e. in $\mathbb{R}$,$\text{ for all } x_0>0$), then $u(t,\cdot)$ is nondecreasing for all $t>0$. 
This follows from the fact that both $u(\cdot,\cdot+x_0)$ and $u$ are mild solutions
of \eqref{N} and the first one has a larger or equal initial datum. As a consequence,
$u(\cdot,\cdot+x_0)\geq u$ a.e., as claimed.

\subsection{A maximum principle}

The following is a maximum principle needed in next section to prove the convergence
of solutions of \eqref{N} towards $1$. It is stated here for classical subsolutions, for which 
the proof is very simple. This will suffice for our purposes ---even that we will
need to work a little more and change some initial data to have classical solutions.
Anyhow, in next subsection we prove the same result for mild solutions,
but the proof is more involved.

Recall that $X_\gamma$ is the Banach space defined in subsection~2.1. It is crucial
for our purposes to have this maximum principle in the space $X_\gamma$
containing certain unbounded functions; in the way that we will proceed, $\Cubn$ would not suffice.
However, note that the proposition also holds in $\Cubn=X_0$.

\begin{proposition}\label{nonlocMPdomain}
Let $n\geq 1$, $\alpha\in(0,1)$, $0\leq \gamma < 2\alpha$,
and $p$ be a kernel satisfying \eqref{P1}-\eqref{P2}-\eqref{P3}. 

Let $v\in C^1([0,\infty);X_\gamma)$
satisfy $v(t,\cdot)\in D_\gamma(A)$ for all $t>0$, and let
$c$ be a continuous function in $(0,\infty)\times\R^n$ which is bounded in 
$(0,T)\times\R^n$ for all $T>0$. Assume in addition:
\begin{eqnarray*}
& {\rm a)} & v(0,\cdot) \leq 0 \ \text{ in } \mathbb{R}^n.\\
& {\rm b)} & \text{For all } T>0, \text{ we have }
\limsup_{|x|\to\infty} v(t,x)\leq 0 \text{ uniformly in } t\in [0,T].\\
& {\rm c)} & \text{if }(t,x)\in (0,\infty)\times\R^n \text{ and } v(t,x)>0, \text{ then }
(v_t+Av)(t,x)\leq c(t,x)v(t,x).
\end{eqnarray*}
Then, $v\leq 0$ in all of $(0,\infty)\times\R^n$.
\end{proposition}

\begin{proof}
Since $v\in C([0,\infty);X_\gamma)$, $v$ is a continuous function in $[0,\infty)\times\R^n$.
Arguing by contradiction, assume that $v>0$ somewhere in $[0,T]\times\R^n$,
for some $T>0$. Let 
$$
w(t,x):=e^{-at}v(t,x), \ \text{ where $a$ is a constant such that }
a>\|c\|_{L^\infty((0,T)\times\R^n)}.
$$

We have that $w>0$ somewhere in $[0,T]\times\R^n$. 
By assumption b), $w$ is bounded above in $[0,T]\times\R^n$
and achieves its positive maximum at some point $(t_0,x_0)\in [0,T]\times\R^n$.
By a) we have $t_0>0$. Since $w\in C^1([0,\infty);X_\gamma)$, we have that 
$w(\cdot,x_0)=w(\cdot)(x_0)$ is differentiable in $(0,t_0]$ and achieves its maximum
in this interval at $t_0$. Thus,
\begin{equation}\label{maxt}
w_t (t_0,x_0)\geq 0.
\end{equation}

On the other hand, by hypothesis,
$w(t_0,\cdot)$ belongs to $D_\gamma(A)$ and achieves its maximum
in $\R^n$ at $x_0$. Thus, by \eqref{MP}
\begin{equation*}
Aw (t_0,x_0)\geq 0.
\end{equation*}
{From} this, \eqref{maxt}, and hypothesis c) (note that $v(t_0,x_0)>0$), we deduce
\begin{eqnarray*}
0 &\leq &(w_t+Aw)(t_0,x_0) = e^{-at_0} (v_t+Av)(t_0,x_0) -ae^{-at_0} v(t_0,x_0)\\
&\leq & e^{-at_0} \left\{c(t_0,x_0)-a\right\} v(t_0,x_0) <0
\end{eqnarray*}
since  $v(t_0,x_0)>0$ and $c-a<0$ in $(0,T]\times\R^n$ (recall that $c$ is continuous in $(0,T]\times\R^n$). 
This is a contradiction.
\end{proof}

We will use the previous result in the situation given by the following
two lemmas. In this first one, we will take $\overline{r}(t)=ae^{\nu t}$ 
in our application, with $a$ and $\nu$
positive constants.

\begin{lemma}\label{AR}
Let $n\geq 1$, $\alpha\in(0,1)$, $0\leq \gamma < 2\alpha$,
and $p$ be a kernel satisfying \eqref{P1}-\eqref{P2}-\eqref{P3}. 

Let $v\in C^1([0,\infty);X_\gamma)$
satisfy $v(t,\cdot)\in D_\gamma(A)$ for all $t>0$, and let
$c$ be a continuous function in $(0,\infty)\times\R^n$ which is bounded in 
$(0,T)\times\R^n$ for all $T>0$.
Let $\overline{r}:[0,+\infty)\to[0,+\infty)$ be a continuous function 
and define
$$
\Omega_r=\left\{(t,x)\in (0,\infty)\times\mathbb{R}^n\, :\, |x|<\overline{r}(t)\right\}.
$$ 
Assume in addition:
\begin{eqnarray}
& {\rm a)} & v(0,\cdot)\leq0 \ \text{ in } \mathbb{R}^n.
\label{AR1}\\
& {\rm b)} & v\leq0 \ \text{ in } \left( (0,\infty)\times\mathbb{R}^n\right) 
\setminus \Omega_r.
\label{AR2}\\
& {\rm c)} & v_t+Av\leq c(t,x)v \  \text{ in } \Omega_r.\label{AR3}
\end{eqnarray}
Then, $v\leq 0$ in all of $(0,\infty)\times\mathbb{R}^n$.
\end{lemma}

The lemma follows immediately from Proposition~\ref{nonlocMPdomain}.

For increasing solutions in $\R$, we will use instead the following result. 
Note that here we assume $c\leq 0$. In our future application, 
we will take $\overline{x}(t)=-be^{\sigma' t}$ in the next lemma, with $b$ and $\sigma'$
positive constants.

\begin{lemma}\label{AI}
Let $n=1$, $\alpha\in(0,1)$, $0\leq \gamma < 2\alpha$,
and $p$ be a kernel satisfying \eqref{P1}-\eqref{P2}-\eqref{P3}. 

Let $v\in C^1([0,\infty);X_\gamma)$
satisfy $v(t,\cdot)\in D_\gamma(A)$ for all $t>0$, and let
$c\leq 0$ be a nonpositive continuous function in $(0,\infty)\times\R$ which is bounded in 
$(0,T)\times\R$ for all $T>0$.
Let $\overline{x}:[0,+\infty)\to\R$ be a continuous function, and define
$$
\Omega=\left\{(t,x)\in (0,\infty)\times\mathbb{R}\, :\, x>\overline{x}(t)\right\}.
$$ 
Assume in addition, for some constant $\delta>0$,
\begin{eqnarray}
& {\rm a)} &  v(0,\cdot)\leq 0 \ \text{ in } \mathbb{R}.
\label{AI1}\\
& {\rm b1)} & v\leq 0 \ \text{ in } \left( (0,\infty)\times\mathbb{R}\right) \setminus \Omega.
\label{AI2}\\
& {\rm b2)} & \text{For all } T>0,\ \limsup_{x\to +\infty}{v(t,x)\leq
\delta} \text{ uniformly in } t\in[0,T].
\label{AI3}\\
& {\rm c)} & v_t+Av \leq c(t,x) v  \ \text{ in } \Omega.\label{AI4}
\end{eqnarray}
Then, $v\leq \delta$ in $(0,+\infty)\times\mathbb{R}$.
\end{lemma}

The lemma follows immediately from Proposition~\ref{nonlocMPdomain} applied to $\tilde{v}:=v-\delta$.
It satisfies $\tilde v_t+A\tilde v = v_t+Av \leq c v =c \tilde v +c\delta\leq  c \tilde v$
in $\{\tilde v >0\}$, since $c\leq 0$ and $\{\tilde v >0\}\subset \{v >0\}\subset \Omega$.

\subsection{A Kato type inequality for mild solutions and applications}

With the results in this subsection ---which are not needed to complete the proofs
of our main theorems--- one may treat the initial data in the proofs of our main
theorems as they are, without having to change the data to belong to $D(A)$.
Recall that in the maximum principle of the previous subsection its proof used 
crucially the solution to be classical and belong to $D(A)$.
In this section we establish that maximum principle, Proposition~\ref{nonlocMPdomain},
for mild solutions; no assumption on the solution being in $D(A)$ is made.
In addition, the proof in this subsection does not require hypothesis b) 
of Proposition~\ref{nonlocMPdomain} on the limits of $v$ as $|x|\to\infty$.
The statement is the following.

\begin{proposition}\label{nonlocMPmild}
Let $n\geq 1$, $\alpha\in(0,1)$, $0\leq \gamma < 2\alpha$,
and $p$ be a kernel satisfying \eqref{P1}-\eqref{P2}-\eqref{P3}. 

Let $v\in C([0,\infty);X_\gamma)$ be the mild solution of
$v_t+Av=h$ in $(0,\infty)$, $v(0,\cdot)=v_0$, where $v_0\in X_\gamma$ 
and $h\in C([0,\infty);X_\gamma)$. Let
$c$ be a continuous function in $(0,\infty)\times\R^n$ which is bounded in 
$(0,T)\times\R^n$ for all $T>0$. Assume in addition:
\begin{eqnarray*}
& {\rm i)} & v(0,\cdot) \leq 0 \ \text{ in } \mathbb{R}^n.\\
& {\rm ii)} & \text{if }(t,x)\in (0,\infty)\times\R^n \text{ and } v(t,x)>0, \text{ then }
h(t,x)\leq c(t,x)v(t,x).
\end{eqnarray*}
Then, $v\leq 0$ in all of $(0,\infty)\times\R^n$.
\end{proposition}

{From} this result, to be proved later in this subsection, one deduces the analogues
of Lemmas \ref{AR} and \ref{AI} for mild solutions in the same way as in the
previous subsection.

To prove Proposition \ref{nonlocMPmild}, we need to establish an inequality of Kato type
for mild solutions. In the stationary case and for functions in the domain of $A$ it states
the following:  
\begin{equation}\label{kato1}
\left \{ \begin{array}{l}
\text{if } \varphi:\R\to\R \text{ is $C^1$ and convex, }v\in D_\gamma(A), \text{ and }
\varphi(v) \in D_\gamma(A),
\vspace{.5mm}
\\
\text{then }A\varphi(v)\leq \varphi'(v) Av \,\text{ in $\R^n$.}
\end{array}\right.
\end{equation}
Its proof is simple. First notice that, by Jensen's inequality,
\begin{equation*}
\begin{split}
(T_s\,\varphi(v) )(x) &= \int_{\R^n} p(s,x-y) \varphi(v(y))\, dy
\\& \geq  \varphi \left( \int_{\R^n} p(s,x-y)v(y)\, dy\right) =\varphi (T_sv (x))
\end{split}
\end{equation*}
and therefore
\begin{equation}\label{calculkato2}
(T_s\varphi(v)-\varphi(v))(x)  \geq 
\varphi (T_sv (x)) -\varphi(v(x)) 
\geq \varphi'(v(x))\, (T_sv-v) (x).
\end{equation}
for all $s>0$. Dividing by $s$ and taking the limits as $s\to 0$ (which we assume to exist), 
we deduce \eqref{kato1}.

When $A=-\Delta$ and $v\in L^1$ is a distributional solution of $-\Delta v=h$,
\eqref{kato1} was first proved by Kato.

The following result states the analogue of \eqref{kato1} for mild solutions in the spaces
$X_\gamma$. Recall that $X_0=\Cubn$; in this space we simply ask the function
$\varphi$ to be $C^1$ and convex. Instead, for $0<\gamma<2\alpha$,  in addition we need to assume
that $\varphi'$ is bounded in $\R$. This is to ensure that
$\varphi(v)\in X_\gamma$ whenever $v\in X_\gamma$ ---recall the functions in $X_\gamma$
may be unbounded if $\gamma >0$. Instead, $\varphi$ being $C^1$ and convex
suffices to ensure that $\varphi(v)\in\Cubn$ whenever $v\in\Cubn$.

\begin{proposition}\label{kato}
Let $n\geq 1$, $\alpha\in(0,1)$, $0\leq \gamma < 2\alpha$,
and $p$ be a kernel satisfying \eqref{P1}-\eqref{P2}-\eqref{P3}. 

Let $0<T\leq +\infty$ and $v\in C([0,T);X_\gamma)$ be the mild solution of
$v_t+Av=h$ in $[0,T]$, $v(0,\cdot)=v_0$, where $v_0\in X_\gamma$ 
and $h\in C([0,T];X_\gamma)$.
Let $\varphi:\R\to\R$ be a $C^1$ convex function. If $\gamma>0$ assume in addition that
$\varphi'$ is bounded.

Then, $\varphi(v) \in C([0,T);X_\gamma)$ satisfies
$\varphi(v)_t+A\varphi(v)\leq \varphi'(v) h$ in the following mild sense:
\begin{equation} \label{convexmild}
\varphi(v(t)) \leq T_t\,\varphi(v_0) +\int_0^t T_{t-s}\, \{\varphi'(v(s))h(s)\} \, ds
\quad\text{ in $\R^n$ for all } t\in [0,T].
\end{equation}
\end{proposition}

\begin{remark}\label{integrable}
{\rm
When $\gamma=0$ and thus $X_0=\Cubn$, we have that $\varphi'(v) h\in C([0,T);\Cubn)$
(simply use that $\varphi'(v)$ is uniformly continuous since $v$ is bounded)
and \eqref{convexmild} is all understood in $\Cubn$. When $0<\gamma<2\alpha$, even
if $\varphi'$ is bounded,
$\varphi'(v) h$ might not verify \eqref{ucgamma} and hence not belong to $X_\gamma$. However, 
$|\varphi'(v) h|\leq C |h|$ for some constant $C$ and thus 
$$
T_{t-s}\, \{\varphi'(v(s))h(s)\}= \int_{\R^n} p(\cdot -y)\varphi'(v(s,y))h(s,y)\, dy
$$
is well defined since $|p(\cdot -y)\varphi'(v(s,y))h(s,y)|$ is integrable
in $y$; see \eqref{quasixgamma}. The remaining integral in $ds$ is also well defined.
}
\end{remark}

We need to establish the inequalities \eqref{convexmild} from the hypothesis
\begin{equation}\label{hypmild}
v(t) = T_t v_0 +\int_0^t T_{t-s}\, h(s) \, ds\quad\text{ for all } t\in [0,T].
\end{equation}
For this, as usual in Kato type inequalities, we need to regularize the weak (here mild) solution
in an appropriate way taking into account the operator $A$. Recall that, by \eqref{intdomain},
for all $w\in X_\gamma$ and $\delta>0$ we have
\begin{equation*}
w^\delta:=\dashint_0^\delta T_\tau w\, d\tau \, \in D_\gamma(A) \quad\text{ and }\quad 
A w^\delta= \frac{1}{\delta}(w-T_\delta w).
\end{equation*}
In addition, $w^\delta\to w$ in $X_\gamma$ as $\delta\downarrow 0$.

\begin{proof}[Proof of Proposition \ref{kato}]
We use the previous regularization to define, for every $t\in [0,T]$, the functions
$v^\delta(t):= (v(t))^\delta$ and $h^\delta(t):= (h(t))^\delta$. 
Note that $h^\delta \in C([0,T];X_\gamma)$,  $h^\delta(t) \in D_\gamma(A)$ for all $t\in [0,T]$, and 
$$
Ah^\delta=\frac{1}{\delta}(h-T_\delta h) \in C([0,T];X_\gamma)
\subset L^1([0,T];X_\gamma).
$$ 
Since in addition $v^\delta(0)\in D_\gamma(A)$,
Corollary~2.6 in section~4.2 of \cite{Pa} gives the existence of a classical solution $u$ to
\begin{equation}\label{problemreg}
\left \{ \begin{array}{rcll}
u_t + A u & = & h^\delta(t) &\  \ \text{in } (0,T), \\
u(0)& = & v^\delta(0); &
\end{array}\right.
\end{equation}
this is shown verifying that, under the above properties of $h^\delta$, the right hand side
of \eqref{hypmild}, with $h$ replaced by $h^\delta$ and $v_0$ by $v^\delta(0)$, is $C^1$ in $t$.
Thus, $u\in C([0,T];X_\gamma)\cap C^1([0,T);X_\gamma)$ satisfies $u(t)\in D_\gamma(A)$
for all $t\in [0,T)$ and \eqref{problemreg} is satisfied pointwise in $[0,T)$.
In particular, $u$ is the mild solution of \eqref{problemreg}. But applying
$\dashint_0^\delta d\tau\, T_\tau$ on equation \eqref{hypmild}, we see that 
$v^\delta$  is the mild solution of \eqref{problemreg}. Thus, $v^\delta=u$ solves
\eqref{problemreg} in the classical sense; in particular
\begin{equation}\label{eqdelta}
(v^\delta)_t + A v^\delta  =  h^\delta \  \ \text{in } (0,T). 
\end{equation}

Since $\varphi (v^\delta(t))\in X_\gamma$ for all $t$ (as discussed in Remark \ref{integrable}), 
we can define
\begin{equation}\label{eqeps}
\varphi^{\delta,\varepsilon}(t):=\{\varphi(v^\delta(t))\}^\varepsilon 
=\dashint_0^\varepsilon T_\tau\,  \varphi(v^\delta(t))\, d\tau
\end{equation}
for $\delta$ and $\varepsilon$ positive. We apply \eqref{calculkato2} with $v$ replaced 
by $v^\delta$ and obtain
\begin{equation*}
T_s\,\varphi(v^\delta)-\varphi(v^\delta) \geq 
\varphi'(v^\delta) \, (T_s v^\delta-v^\delta).
\end{equation*}
As pointed out in Remark \ref{integrable}, $\varphi'(v^\delta) \, (T_s v^\delta-v^\delta)$
could not belong to $X_\gamma$ when $\gamma>0$. However, its absolute value is bounded by
$C|T_s v^\delta-v^\delta|$, which satisfies \eqref{bddgamma} and thus we may act the
convolution semigroup on this function.
Applying $\dashint_0^\eps d\tau\, T_\tau$ on the previous inequality 
and dividing by $s$, we deduce
\begin{equation*}
\frac{T_s\varphi^{\delta,\eps}-\varphi^{\delta,\eps}}{s} \geq 
\dashint_0^\eps T_\tau\left\{\varphi'(v^\delta) \, \frac{T_s\,v^\delta-v^\delta}{s}\right\}
\,d\tau.
\end{equation*}
We now let $s\downarrow 0$ (also use that $\varphi'(v^\delta)$ is bounded and that
$(T_s\,v^\delta-v^\delta)/s$ converges in $X_\gamma$) to deduce
$$
A\varphi^{\delta,\eps}\leq 
\dashint_0^\eps T_\tau\left\{\varphi'(v^\delta) \,Av^\delta\right\}\,d\tau
=\dashint_0^\eps T_\tau\left\{\varphi'(v^\delta) \,(h^\delta-(v^\delta)_t)\right\}
\,d\tau ,
$$
where in the last equality we have used \eqref{eqdelta}.

Since $v^\delta(t)$ is differentiable in $t$, 
the right hand side of \eqref{eqeps} also is differentiable in $t$
and we have 
$(\varphi^{\delta,\eps})_t=\dashint_0^\eps T_\tau\left\{\varphi'(v^\delta) (v^\delta)_t
\right\}\,d\tau$.
Adding this to the previous inequality and defining 
\begin{equation}\label{newhomog}
(\varphi^{\delta,\eps})_t + A\varphi^{\delta,\eps}=: g_{\delta,\eps}, 
\end{equation}
we find
$$
(\varphi^{\delta,\eps})_t + A\varphi^{\delta,\eps}= g_{\delta,\eps}\leq 
\dashint_0^\eps T_\tau\left\{\varphi'(v^\delta) h^\delta\right\}
\,d\tau .
$$
Hence, since \eqref{newhomog} also holds in the mild sense, we have
\begin{equation*}
\begin{split}
\varphi(v^{\delta,\eps}(t)) &= T_t\,\varphi(v^{\delta,\eps}(0)) +
\int_0^t ds\, T_{t-s} g_{\delta,\eps}(s)\\
&\leq T_t\,\varphi(v^{\delta,\eps}(0)) +
\int_0^t ds\, T_{t-s} \dashint_0^\eps d\tau\,T_\tau
\{\varphi'(v^\delta(s))h^\delta(s)\} 
\end{split}
\end{equation*}
in $\R^n$ for all $t\in [0,T]$. Finally, since $\varphi'(v^\delta(s))h^\delta(s)\in C(\R^n)$,
letting $\eps\downarrow 0$ we deduce (pointwise in $\R^n$)
$$
\varphi(v^{\delta}(t)) \leq T_t\,\varphi(v^{\delta}(0)) +
\int_0^t ds\, T_{t-s} \{\varphi'(v^\delta(s))h^\delta(s)\}.
$$
Letting $\delta\downarrow 0$ and using dominated convergence, we conclude
$$
\varphi(v(t)) \leq T_t\,\varphi(v_0) +
\int_0^t ds\, T_{t-s} \{\varphi'(v(s))h(s)\}.
$$
This is the statement \eqref{convexmild} of the proposition.
\end{proof}

Using the proposition we can now prove the maximum principle for mild solutions.

\begin{proof}[Proof of Proposition \ref{nonlocMPmild}]
Let $\varphi:\R\to\R$ be a $C^1$ convex function such that 
$$
\varphi\equiv 0 \ \text{in } (-\infty,0),\ \varphi >0 \ \text{in } (0,+\infty),
\text{ and } 0 \leq \varphi'\leq 1 \text{ in } \R.
$$
For instance, we may take $\varphi\equiv 0$ in $(-\infty,0)$ and
$\varphi (u)= \frac{u^2}{u+1}$ in $[0,+\infty)$.

Since $v\in C([0,\infty);X_\gamma)$, $v$ is a continuous function in $[0,\infty)\times\R^n$.
Arguing by contradiction, assume that $v>0$ somewhere in $[0,T]\times\R^n$,
for some $T>0$. Let 
$$
w(t,x):=e^{-at}v(t,x), \ \text{ where $a$ is a constant such that }
a\geq \|c\|_{L^\infty((0,T)\times\R^n)}.
$$

Since  $v\in C([0,T);X_\gamma)$ is the mild solution of
$v_t+Av=h(t)$ in $[0,T]$, $v(0,\cdot)=v_0$, \eqref{newu}-\eqref{Nga}-\eqref{defga}
give that  $w\in C([0,T);X_\gamma)$ is the mild solution of
$w_t+Aw=e^{-at}\{ -av(t)+h(t)\}$ in $[0,T]$, $w(0,\cdot)=v_0$.

Therefore, by Proposition \ref{kato}, we have that
\begin{equation}\label{finalkato}
\varphi(w(t)) \leq T_t\,\varphi(v_0) +\int_0^t T_{t-s}\, 
\{\varphi'(w(s))e^{-as}(-av(s)+h(s))\} \, ds
\end{equation}
in $\R^n$ for all $t\in [0,T]$. But $v_0\leq 0$ by hypothesis i) in the proposition,
and thus $\varphi(v_0)\equiv 0$.
In addition, $\varphi'(w(s))(x)= 0$ whenever $w(s)(x)\leq 0$. If $w(s)(x)> 0$, 
then also $v(s)(x)>0$ and by hypothesis ii), we have $h(s,x)\leq c(s,x)v(s,x) \leq
av(s,x)$, and thus $-av(s,x)+h(s,x)\leq 0$. Finally, $\varphi'(w(s))\geq 0$ in all of $\R^n$.

We conclude that $\varphi'(w(s))e^{-as}(-av(s)+h(s))\leq 0$ in all of $\R^n$,
and by \eqref{finalkato} that $\varphi(w(t))\leq 0$ in $\R^n$ for all $t\in [0,T]$.
This leads to $w(t) \leq 0$, and thus  $v(t) \leq 0$ in $\R^n$ for all $t\in [0,T]$.
This contradicts our initial assumption: $v>0$ somewhere in $[0,T]\times\R^n$.
\end{proof}

\subsection{Bounds on the semigroup}

Next, a well-known simple lemma. For completeness, we include its proof below. 

\begin{lemma}\label{DECR}
Let $u\in L^{1}(\mathbb{R}^n)$ and $v\in L^{\infty}(\mathbb{R}^n)$ be positive, 
radially symmetric, and nonincreasing functions
in $\mathbb{R}^n$, where $u\in C^1$ has its radial derivative $u'\in L^{1}(\mathbb{R}^n)$.
Then, $u\ast v$ is 
also positive, radially symmetric, and nonincreasing. 
\end{lemma}

\begin{proof} 
Denote the convolution by
$$
w(x):=\int_{\mathbb{R}^n}u(|x-y|)v(|y|)\, dy,$$
clearly a positive and radially symmetric function. We compute
\begin{equation*}
\nabla w(x)\cdot x = \int_{\mathbb{R}^n} u'(|x-y|)\frac{(x-y)\cdot x}
{|x-y|} v(|y|)\, dy
=\int_{\mathbb{R}^n} u'(|z|)\frac{z\cdot x}{|z|}
v(|x-z|)\, dz.
\end{equation*}
In $\left\{x\cdot z\leq0\right\}$ we make the change $\xi=-z$ and obtain
$$
\int_{\left\{x\cdot z\leq0\right\}} u'(|z|)\frac{x\cdot z}
{|z|} v(|x-z|)\, dz = 
\int_{\left\{x\cdot \xi\geq0\right\}} u'(|\xi|)\frac{-x\cdot 
\xi}{|\xi|} 
v(|x+\xi|)\,d\xi.
$$
Thus,
$$
\nabla w(x)\cdot x = \int_{\left\{x\cdot z\geq0\right\}} u'(|z|)\frac{x\cdot 
z}{|z|} \left\{v(|x-z|)-v(|x+z|)\right\}\, dz. $$
We conclude noticing that the first factor is nonpositive, while the second and third 
are nonnegative since 
$v$ is radially nonincreasing 
and $|x-z|^2\leq|x+z|^2$ in the set $\{x\cdot z\geq0\}$.
\end{proof}

The next lemma will help us handle the $\Con$ initial data in Theorem~\ref{LIMR}.

\begin{lemma}\label{HR} 
Let $n\geq 1$, $\alpha\in(0,1)$, $\gamma\in (0,2\alpha)$,
and $p$ be a kernel satisfying \eqref{P1}-\eqref{P2}-\eqref{P3}.
Recall that $B$ is the constant in \eqref{P3}.
Then, for some positive constants $c$, $C$, $c_\gamma$, and  $C_\gamma$ depending only on
$n$, $\alpha$, and $B$, and also on $\gamma$ in the case of $c_\gamma$ and $C_\gamma$, we have:

{\rm a)} Let $a_0>0$, $r_0\geq 1$, and 
    \begin{equation*}
    v_0 (x)= \left \{ \begin{array}{lcl}
    a_0 |x|^{-n-2\alpha} & \text{ for} & |x|\geq r_0,\\
    a_0 r_0^{-n-2\alpha} & \text{ for} & |x|\leq r_0.
    \end{array}\right.
    \end{equation*}
Then, 
$$
T_t v_0(x)\leq C(1+r_0^{-2\alpha}t)a_0 |x|^{-n-2\alpha}
\quad\text{ for all } t>0,\, x\in\mathbb{R}^n,
$$
and
\begin{eqnarray*}
T_t v_0(x) &\geq & B^{-1} (q(t,\cdot)\ast v_0)(x)\\
& \geq &
c\, \frac{t}{t^{\frac{n}{2\alpha}+1}+1}a_0 
|x|^{-n-2\alpha} \quad\text{ if }  t>0,\, |x|\geq r_0,
\end{eqnarray*}
where $q$ is the function defined in \eqref{nameq}. 

{\rm b)} Let $w_\gamma(x)=|x|^{\gamma}$. Then, 
$$
T_t w_\gamma(x)\leq C_{\gamma}(|x|^{\gamma}+t^{\frac{\gamma}{2\alpha}})
\quad\text{ for all } t>0,\, 
x\in\mathbb{R}^n,
$$
and
$$
T_t w_\gamma(x)\geq c_{\gamma}|x|^{\gamma}\quad\text{ if } t>0,\, 
|x|\geq t^{\frac{1}{2\alpha}}.
$$
\end{lemma}

\begin{proof}
We start proving a). The quantity
  $T_tv_0(x)$ is comparable, up to multiplicative constants, to the integral
    \begin{equation}\label{intalpha}
 I:=   \int_{\mathbb{R}^n}\frac{t^{-\frac{n}{2\alpha}}}{1+(t^{-\frac{1}{2\alpha}}|x-y|)^
{n+2\alpha}}\, v_0(y)\, dy.
    \end{equation}
   We start with the upper bound. In \eqref{intalpha} we integrate first in 
$B_{|x|/2}(0)$ and then in $\mathbb{R}^n\setminus B_{|x|/2}(0)$.
    In $B_{|x|/2}(0)$, we have $|x-y|\geq |x|-
|y|\geq |x|/2$, and thus the integral is bounded above by
    $$I_1:= \frac{t^{-\frac{n}{2\alpha}}}
    {1+(t^{-\frac{1}{2\alpha}} |x|/2)^{n+2\alpha}} 
\int_{B_{|x|/2}(0)}\,v_0(y)\,dy.$$
    Now, $\int_{B_{r_0}(0)} v_0(y)\, dy = a_0\,r_0^{-n-2\alpha}Cr_0^n=Ca_0\,r_0^{-2\alpha}$.
In case $r_0<|x|/2$, the remaining term in the integral 
over $B_{|x|/2}(0)$ is also estimated by
    $$\int_{B_{|x|/2}(0)\setminus B_{r_0}(0)}v_0(y)\, dy = C\int_{r_0}^{|x|/2}a_0
\,r^{-n-2\alpha}r^{n-1}dr\leq  Ca_0\,r_0^{-2\alpha}.
$$
Hence, 
  \begin{equation}\label{firstpart}
 I_1\leq \frac{t^{-\frac{n}{2\alpha}}}{(t^{-\frac{1}{2\alpha}} |x|/2)^
{n+2\alpha}}Ca_0\,r_0^{-2\alpha}
  =Ct\,r_0^{-2\alpha}a_0|x|^{-n-2\alpha}.
  \end{equation}
  For the integrand in \eqref{intalpha} over 
$\mathbb{R}^n\backslash B_{|x|/2}(0)$, note that $v_0(y)\leq 
a_0(|x|/2)^{-n-2\alpha}$ in this set. Thus, the integral over this set is bounded above by
  $$Ca_0|x|^{-n-2\alpha}\int_{\mathbb{R}^n}\frac{t^{-\frac{n}{2\alpha}}}
{1+(t^{-\frac{1}{2\alpha}} |x-y|)^{n+2\alpha}}\, dy=Ca_0|x|^{-n-2\alpha}
\int_{\mathbb{R}^n}\frac{d\overline{y}}{(1+|\overline{y}|)^{n+2\alpha}}.$$
Therefore, we have the upper bound $Ca_0|x|^{-n-2\alpha}$. 

Putting this together with \eqref{firstpart}, we conclude  
  $$T_tv_0(x)\leq C(1+r_0^{-2\alpha}t)a_0|x|^{-n-2\alpha}.$$
  
Next, we show the lower bound. We assume $|x|\geq r_0\geq1$. We have
  $$T_tv_0(x)\geq B^{-1} \int_{B_1(x)}\frac{t^{-\frac{n}{2\alpha}}}{1+(t^{-\frac{1}{2\alpha}} |x-y|)^
{n+2\alpha}}\,v_0 (y) \, dy.$$
In the set of integration $|y|\leq |x|+1\leq |x|+r_0\leq 
2|x|$, and thus 
  $v_0(y)\geq v_0(2|x|)=a_0(2|x|)^{-n-2\alpha}$. Finally, since
  $$\int_{B_1(0)}\frac{t}{t^{\frac{n}{2\alpha}+1}+|z|^{n+2\alpha}}\,dz
  \geq \frac{t}{t^{\frac{n}{2\alpha}+1}+1}\int_{B_1(0)}dz,$$
we conclude the statement in the lemma.
  
We now prove part b). The quantity $T_tw_\gamma(x)$ is comparable to 
  \begin{equation}\label{intgamma}
  \int_{\mathbb{R}^n}\frac{t^{-\frac{n}{2\alpha}}}{1+(t^{-\frac{1}{2\alpha}}|x-y|)^
{n+2\alpha}}|y|^{\gamma}\, dy.
  \end{equation} 
  For the upper bound, we make the change of variables $\overline{y}=t^{-\frac{1}{2\alpha}}(x-y)$
and notice that 
$|y|^{\gamma}
  \leq (|x|+t^{\frac{1}{2\alpha}}|\overline{y}|)^{\gamma}$. Thus 
\eqref{intgamma} is smaller than 
  $$C_\gamma \int_{\mathbb{R}^n}\frac{1}{1+|\overline{y}|^{n+2\alpha}}(|x|^{\gamma}+t^
{\frac{\gamma}{2\alpha}}|\overline{y}|^{\gamma})\, d\overline{y}\leq
  C_\gamma (|x|^{\gamma}+t^{\frac{\gamma}{2\alpha}})$$ 
  since $\gamma<2\alpha$. 

For the lower bound, we 
  assume $|x|\geq t^{\frac{1}{2\alpha}}$. We estimate \eqref{intgamma} from below by the 
same 
integral in $y\in B_{|x|/2}(x)$. Here, 
  $|y|\geq |x|-|x|/2=|x|/2$. 
Making  the change of variables $\overline{y}=t^{-\frac{1}{2\alpha}}(x-y)$, 
  we minorize \eqref{intgamma} by
  $$\left(|x|/2\right)^{\gamma}\int_{\{|\overline{y}|<t^{-\frac{1}{2\alpha}}|x|/2\}}
\frac{d\overline{y}}{1+|\overline{y}|^{n+2\alpha}}.$$
  Since $|x|\geq t^{\frac{1}{2\alpha}}$ by hypothesis, the last integral is 
larger or equal than a positive constant.
\end{proof}

The previous lemma has the following counterpart for nondecreasing
initial data in $\R$.

\begin{lemma}\label{HI}
Let $n= 1$, $\alpha\in(0,1)$, $\gamma\in (0,2\alpha)$,
and $p$ be a kernel satisfying \eqref{P1}-\eqref{P2}-\eqref{P3}.
Recall that $B$ is the constant in \eqref{P3}.
Then, for some positive constants $c$, $C$, $c_\gamma$, and $C_\gamma$ depending only on
$\alpha$ and $B$, and also on $\gamma$ in the case of $c_\gamma$ and $C_\gamma$, we have:

{\rm a)} Let $a_0>0$ and $x_0\leq-1$. Let 
    \begin{equation*} 
    V_0(x)=\left \{ \begin{array}{lll}
    a_0 |x|^{-2\alpha}   &\text{ for} &x\leq x_0, \\
    a_0 |x_0|^{-2\alpha} &\text{ for} &x\geq x_0.
    \end{array}\right.
    \end{equation*}
    Then,
    $$T_t V_0(x)\leq C (1+|x_0|^
{-2\alpha}t)a_0|x|^{-2\alpha} \quad\text{ if } t>0,\, x<2x_0,$$
and
$$T_t V_0(x)\geq  c\di\frac{t}{t^{\frac{1}{2\alpha}+1}+1}
a_0|x|^{-2\alpha} \quad\text{ if }  t>0,\, x<x_0.$$

{\rm b)}   Let  $W_\gamma(x)=(x_-)^{\gamma}$, where $x_-$ denotes the negative 
part of $x$. Then,
    $$c_{\gamma}(|x|^{\gamma}+t^{\frac{\gamma}{2\alpha}})
\leq T_t W_\gamma(x) \leq C_{\gamma} (|x|^{\gamma}+t^{\frac{\gamma}{2\alpha}}) \quad\text{ if }
 t>0,\, x<0.$$
\end{lemma}

\begin{proof}

We start proving a). First, the upper bound.  Consider $x<2x_0<0$, then
\begin{eqnarray*} 
T_tV_0(x)&\leq& B\int^{+\infty}_{-\infty}\frac{t^{-\frac{1}{2\alpha}}}
{1+(t^{-\frac{1}{2\alpha}} |x-y|)^{1+2\alpha}}V_0(y)\, dy\\
&=& B \int^{x/2}_{-\infty}\frac{t^{-\frac{1}{2\alpha}}}
{1+(t^{-\frac{1}{2\alpha}}|x-y|)^{1+2\alpha}}V_0(y)\, dy+ 
\\
&& \hspace{1cm}
 + B\int^{+\infty}_{x/2}\frac{t^{-\frac{1}{2\alpha}}}{1+(t^{-\frac{1}{2\alpha}}|x-y|)^
{1+2\alpha}}V_0(y)\,dy\\
&\leq & Ca_0|x/2|^{-2\alpha}
\int^{x/2}_{-\infty}\frac{t^{-\frac{1}{2\alpha}}\,dy}{1+(t^{-\frac{1}{2\alpha}}|x-y|)
^{1+2\alpha}}+
\\
&& \hspace{1cm}
+C a_0\vert x_0\vert^{-2\alpha}\int^{+\infty}_{x/2}
\frac{t^{-\frac{1}{2\alpha}}\,dy}{1+(t^{-\frac{1}{2\alpha}}|x-y|)^
{1+2\alpha}}.
\end{eqnarray*}
We conclude by noticing that
$$\int_{-\infty}^{x/2}\frac{t^{-\frac{1}{2\alpha}}\, dy}{1+(t^{-\frac{1}{2\alpha}}|x-y|)
^{1+2\alpha}}\leq \int_{-\infty}^{+\infty}
\frac{d\overline{y}}{1+\overline{y}^{1+2\alpha}}=C
$$
and
$$\int^{+\infty}_{x/2}\frac{t^{-\frac{1}{2\alpha}}\, dy}{1+(t^{-\frac{1}{2\alpha}}|x-y|)
^{1+2\alpha}}=\int^{+\infty}_{-t^{-1/(2\alpha)}x/2}
\frac{d\overline{y}}{1+\overline{y}^{1+2\alpha}}\leq Ct|x|^{-2\alpha}.
$$

Next, the lower bound. Since $x<x_0\leq -1$, we have
\begin{eqnarray*} 
T_tV_0(x)
&\geq&B^{-1}\int^{x}_{x-1}\frac{t^{-\frac{1}{2\alpha}}}
{1+(t^{-\frac{1}{2\alpha}}|x-y|)^{1+2\alpha}}\frac{a_0}
{|y|^{2\alpha}}\,dy\\
&\geq&B^{-1}\frac{a_0}{(2|x|)^{2\alpha}}\int^{1}_{0}\frac{t^{-\frac{1}{2\alpha}}}
{1+(t^{-\frac{1}{2\alpha}}z)^{1+2\alpha}}\, dz,
\end{eqnarray*}
where we have used $|y|=-y\leq1-x\leq-x-x=-2x=2|x|$ 
in the last bound. Finally, using $0\leq z\leq 1$ in the last integral,
we conclude the lower bound.

We now prove b). The upper bound is a consequence of the upper
bound in part~b) of Lemma~\ref{HR}. For the lower bound, since $x<0$ note that 
\begin{eqnarray*} 
T_tW_\gamma(x)&\geq& B^{-1}\int^{0}_{-\infty}\frac{t^{-\frac{1}{2\alpha}}}{1+(t^{-\frac{1}{2\alpha}}|x-y|)
^{1+2\alpha}}|y|^{\gamma}\,dy\\
&\geq &
 B^{-1}\int^{x-t^{\frac{1}{2\alpha}}}_{x-(2t)^{\frac{1}{2\alpha}}}
\frac{t^{-\frac{1}{2\alpha}}}{1+(t^{-\frac{1}{2\alpha}}|x-y|)
^{1+2\alpha}}|y|^{\gamma}\,dy\\
&\geq &
c_\gamma\int^{x-t^{\frac{1}{2\alpha}}}_{x-(2t)^{\frac{1}{2\alpha}}}
t^{-\frac{1}{2\alpha}}|y|^{\gamma}\,dy
\geq c_{\gamma} |x-t^{\frac{1}{2\alpha}}|^\gamma
\geq c_{\gamma}(|x|^{\gamma}+t^{\frac{\gamma}{2\alpha}})
\end{eqnarray*}
since $|x-t^{\frac{1}{2\alpha}}|=|x|+t^{\frac{1}{2\alpha}}$. This concludes the proof.
\end{proof}

\section{Initial data with compact support}

To prove part b) (the convergence towards $1$) of Theorem \ref{LIMR}, 
we will need the following key lemma.

\begin{lemma}\label{EXPR}
Let $n\geq 1$, $\alpha\in(0,1)$, $f$ satisfy \eqref{nonl},
and $p$ be a kernel satisfying \eqref{P1}-\eqref{P2}-\eqref{P3}.
Recall that $B$ is the constant in \eqref{P3}. Then, for every 
$0<\sigma<\frac{f'(0)}{n+2\alpha}$, 
there exist $t_0\geq1$ and $0<\varepsilon_0 <1$ depending only on $n$, 
$\alpha$, 
$B$, $f$, and $\sigma$, for which the following holds. 

Given $r_0\geq 1$ and $0<\varepsilon\leq\varepsilon_0$, let $a_0>0$ be defined by 
$a_0r_0^{-n-2\alpha}=\varepsilon$ and let
\begin{equation*} 
v_0(x)=\left \{ \begin{array}{lll}
a_0 |x|^{-n-2\alpha}   &\text{for} &|x|\geq r_0, \\
\varepsilon = a_0 r_0^{-n-2\alpha} &\text{for} &|x|\leq r_0.
\end{array}\right.
\end{equation*}
Then, the mild solution $v$ of \eqref{N} with initial condition $v_0$ satisfies
$$v(kt_0,x)\geq \varepsilon \ \ \ \text{for }|x|\leq r_0e^{\sigma kt_0}$$
and $k\in\left\{0,1,2,3,\ldots\right\}$.
\end{lemma}

\begin{proof}
The lemma being of course true for $k=0$, let us prove it for $k=1$.  
Let $\delta\in(0,1)$ be sufficiently small such that
\begin{equation}\label{deltaR}
\sigma<\frac{1}{2}\left(\sigma+\frac{f'(0)}{n+2\alpha}\right)<\frac{1}{n+2\alpha}
\frac{f(\delta)}{\delta}<\frac{f'(0)}{n+2\alpha}.
\end{equation}
We take $t_0\geq 1$ sufficiently large, 
depending only on 
$n$, $\alpha$, $B$, $f$ and $\sigma$, such that
\begin{equation}\label{deft0}
\left(c\frac{t_0}{t_0^{\frac{n}{2\alpha}+1}+1}\right)^\frac{1}{n+2\alpha} 
e^{\frac{1}{2}\left(\sigma+\frac{f'(0)}{n+2\alpha}\right)t_0}
\geq e^{\sigma t_0},
\end{equation}
where $c>0$ is the constant in the lower bound in part~a) of Lemma~\ref{HR}.
In particular, $c$ depends only on $n$, $\alpha$, and $B$.
Define now $0<\varepsilon_0<\delta$ by
\begin{equation*}
\varepsilon_0 = \delta e^{-f'(0)t_0}.
\end{equation*}

Recall that, in what follows, we are given $r_0\geq 1$ and $\varepsilon$ such that
$$
0<\varepsilon\leq\varepsilon_0<\delta.
$$
Let  
$$
w:=e^{(f(\delta)/\delta)t}\, T_tv_0.
$$ 
It satisfies 
$$
w_t+Aw=\frac{f(\delta)}{\delta}w,\ \ w(0,\cdot)=v_0
$$
in the mild sense. Since $v_0\leq\varepsilon$ in $\R^n$, we also have
$T_tv_0\leq\varepsilon$ in $\R^n$ for all $t>0$.
Now, for $t\leq t_0$, $0\leq w \leq e^{(f(\delta)/\delta)t_0}\varepsilon\leq 
e^{f'(0)t_0}\varepsilon_0=\delta$.
Since $(f(\delta)/\delta)w\leq f(w)$ for $0\leq w\leq\delta$, we have that $w$ 
is a mild subsolution of \eqref{N} in $\left[0,t_0\right]\times\mathbb{R}^n$. 
Thus, by the comparison principle of subsection~2.3, we have
\begin{equation}\label{allRn}
v(t_0,\cdot)\geq w(t_0,\cdot)\geq \underline{w}(t_0,\cdot)
\ \ \text{in}\ \mathbb{R}^n,
\end{equation}
where 
$$
\underline{w}(t,x):=B^{-1}e^{(f(\delta)/\delta)t} (q(t,\cdot)\ast v_0)(x)
$$
and $q$ was defined in \eqref{nameq}. We will use that 
$\underline{w}(t,\cdot)$ is radially nonincreasing by Lemma \ref{DECR}.

By the lower bound in part a) of Lemma \ref{HR}, we have 
\begin{equation}\label{powR}
v(t_0,x)\geq w(t_0,x)\geq \underline{w}(t_0,x)
\geq e^{(f(\delta)/\delta)t_0}c\frac{t_0}{t_0^{\frac{n}{2\alpha}+1}+1}a_0|x|^{-n-2\alpha}
\ \ \text{ for}\ |x|\geq r_0. 
\end{equation}
Let us define $r_1>0$ by 
\begin{equation}\label{defx1R}
e^{(f(\delta)/\delta)t_0}c\,\frac{t_0}{t_0^{\frac{n}{2\alpha}+1}+1}\,\frac{a_0}
{r_1^{n+2\alpha}}=\varepsilon.
\end{equation}
Since $a_0=\varepsilon r_0^{n+2\alpha}$, we get
$$r_1=r_0\left(c\frac{t_0}{t_0^{\frac{n}{2\alpha}+1}+1}\right)^\frac{1}{n+2\alpha} 
e^{\frac{1}{n+2\alpha}(f(\delta)/\delta)t_0}.$$
By \eqref{deft0} and the second inequality in \eqref{deltaR}, we have
\begin{equation}\label{bddx1R}
r_1\geq r_0 e^{\sigma t_0}>r_0.
\end{equation}

Now, since $r_1>r_0$, \eqref{powR} and \eqref{defx1R} lead to $v(t_0,x)\geq\underline{w}
(t_0,x)\geq a_1|x|^{-n-2\alpha}$ for $|x|\geq r_1$, where $a_1:=\varepsilon r_1^{n+2\alpha}$. 
Since $\underline{w}$ is radially nondecreasing by 
Lemma \ref{DECR}, \eqref{allRn}-\eqref{powR}-\eqref{defx1R} lead to $v(t_0,x)\geq\underline{w}(t_0,x)\geq 
\underline{w}(t_0,r_1)\geq \varepsilon$ 
for $|x|\leq r_1$. 

Thus, $v(t_0,\cdot)\geq {v}_1$ where ${v}_1$ 
is given by the expression for $v_0$ in the statement of the lemma with 
$(r_0,a_0)$ 
replaced by $(r_1,a_1)$. Note that $r_1\geq r_0\geq 1$.

Therefore, we can repeat the argument above successively, now with initial times 
$t_0, 2t_0, 3t_0,\ldots$ and radius $r_1, r_2, r_3,\ldots$, and obtain
$$
v(kt_0,x)\geq\varepsilon\ \ \text{ for } |x|\leq r_k,
$$
for all $k\in\left\{0,1,2,3,\ldots\right\}$. Since
$$
r_k\geq r_0 e^{\sigma kt_0}
$$
by \eqref{bddx1R}, the statement of the lemma follows.
\end{proof}

\begin{corollary}\label{BDR}
Let $n\geq 1$, $\alpha\in(0,1)$, $f$ satisfy \eqref{nonl},
$p$ be a kernel satisfying \eqref{P1}-\eqref{P2}-\eqref{P3}, and
$0<\sigma<\frac{f'(0)}{n+2\alpha}$. Let $t_0\geq 1$ be the time given by Lemma~\ref{EXPR}. 

Then, for every measurable initial datum $u_0$ with $0\leq u_0\leq1$ and $u_0\not\equiv 0$, 
there exist $\varepsilon\in(0,1)$ 
and $b>0$ (both depending on $u_0$)
such that 
$$
u(t,x)\geq\varepsilon \quad\ \text{ for all } t\geq t_0 \text{ and }
|x|\leq be^{\sigma t},
$$
where $u $ is the mild solution of $\eqref{N}$ with $u(0,\cdot)=u_0$.
\end{corollary}

\begin{proof}
Since $u$ is a supersolution of the homogeneous problem (the problem with $f=0$), we have 
that $u(t_0/2,\cdot)\geq T_{t_0/2}u_0 > 0$ in $\mathbb{R}^n$, 
since $u_0\not\equiv0$. Thus, since $T_{t_0/2}u_0$ is a positive continuous function in $\mathbb{R}^n$,
we have $u(t_0/2,\cdot)\geq \eta \chi_{B_1(0)}$ in $\mathbb{R}^n$
for some constant $\eta>0$. Therefore, 
\begin{equation}\label{inequw}
u(t_0/2+t,\cdot)\geq T_t\left(\eta\chi_{B_1(0)}\right)
\geq \underline{v}(t,\cdot):=B^{-1}\eta \, q(t,\cdot)\ast\chi_{B_1(0)}
\ \ \text{in}\ \mathbb{R}^n,
\end{equation}
where $q$ was defined in \eqref{nameq}. We will use that 
$\underline{v}(t,\cdot)$ is radially nonincreasing by Lemma \ref{DECR}.

To bound $\underline{v}$ by below, we use the second inequality in \eqref{lowercharR} with
$t\in\left[t_0/2,3t_0/2\right]$. We take $x\in\mathbb{R}^n$ with 
$|x|\geq t_0^{\frac{1}{2\alpha}}\geq 1$ to have 
$t^{\frac{n}{2\alpha}+1}+|x|^{n+2\alpha}\leq 
C|x|^{n+2\alpha}$ 
for such $t$ and~$x$. We deduce
\begin{equation}\label{lbrt}
\underline{v}(t,x)\geq a_0|x|^{-n-2\alpha}\ \ \text{for} \ 
t\in\left[t_0/2,3t_0/2\right]\ \text{and}\ |x|
\geq r_0:=t_0^{\frac{1}{2\alpha}},
\end{equation}
for some $a_0>0$.
We make $a_0$ smaller, if necessary, to have that $\varepsilon:=a_0r_0^{-n-2\alpha}
\leq \varepsilon_0$, where $\varepsilon_0$ is given by Lemma \ref{EXPR}. 
Since $\underline{v}$ is radially nonincreasing, from \eqref{inequw} and \eqref{lbrt} 
we deduce 
$$
u(t_0/2+t,\cdot)\geq \underline{v}(t,\cdot)\geq v_0\ \ \text{ in }\mathbb{R}^n, 
\text{for all}\ t\in\left[t_0/2,3t_0/2\right],
$$
where $v_0$ is the initial condition in Lemma \ref{EXPR}. 

Thus, we can apply Lemma \ref{EXPR} to get a lower bound for $u(\cdot+\tau_0,\cdot)$ 
for all $\tau_0\in[t_0,2t_0]$. Since $\left\{\tau_0+kt_0\, |\, k=0,1,2,\ldots\ \ 
\text{and} \ \tau_0\in[t_0,2t_0]\right\}$ cover all $[t_0,\infty)$, we deduce 
$$
u(t,x)\geq\varepsilon\ \ \text{if}\ \ t\geq t_0\ \ \text{and}\ \ 
|x|\leq r_0 e^{-\sigma2t_0}e^{\sigma t}
$$
by taking $t=\tau_0+kt_0$ and using $|x|\leq r_0 e^{-\sigma2t_0} e^{\sigma t}\leq 
r_0 e^{-\sigma\tau_0}e^{\sigma t}=r_0e^{\sigma kt_0}$. This last statement proves the corollary 
taking $b=r_0e^{-\sigma 2t_0}$. 
\end{proof}

Using Corollary \ref{BDR} we can easily deduce Proposition \ref{NTW} on 
nonexistence of traveling waves. 

\begin{proof}[Proofs of Lemma~\ref{epslemma} and Proposition~\ref{NTW}]
We apply Corollary \ref{BDR} with $\sigma$ replaced by $\sigma'$, where
$\sigma'\in(\sigma,f'(0)/(n+2\alpha))$. Since $e^{\sigma t}\leq b e^{\sigma' t}$ for $t$ 
large (where $b$ is the constant in the statement of Corollary \ref{BDR}), we deduce 
the statement of Lemma~\ref{epslemma}, i.e.,
\begin{equation*}
u(t,x)\geq\varepsilon \ \ \quad \text{ for }  t\geq\underline{t}
\text{ and } |x|\leq e^{\sigma t}.
\end{equation*}

We can now prove Proposition \ref{NTW}. That is, all solutions $u$ of \eqref{N}
with values in $[0,1]$ 
and of the form  $u(t,x)=\varphi (x + t e)$, for some vector $e\in\R^n$, are
identically 0 or 1. 

Indeed, assume that $u\not\equiv 0$ and replace the initial datum $\varphi(x)$ for $u$ by
the smaller one $\min\{\varphi(x),|x|^{-n-2\alpha}\}$. The mild solution for this new initial
condition is smaller than $u$ and satisfies, by Lemma~\ref{epslemma}, the conclusion
of the lemma for any given $\sigma<\sigma_\ast$. Hence, we also have that
$\varphi (x + t e)=u(t,x)\geq\varepsilon$ if $|x|\leq e^{\sigma t}$ and $t\geq \underline{t}$.
As a consequence, $\varphi (y)\geq\varepsilon$ if $|y-te|\leq|y|+t|e|\leq e^{\sigma t}$ and 
$t\geq \underline{t}$.
But, given any $y\in\R^n$, the two last inequalities are true for $t$ large enough.
We deduce that $\varphi\geq\varepsilon$ in all of $\R^n$, and hence
$u\geq\varepsilon$ in all of $(0,\infty)\times\R^n$.

Note now that $f(s)\geq\frac{f(\varepsilon)}{1-\varepsilon}(1-s)$ for all $s\in [\varepsilon,1]$. 
Thus, $u\geq v$ where $v$ is the solution of the linear problem 
\begin{equation*}
\left \{ \begin{array}{rcll}
v_t+Av & = & \frac{f(\varepsilon)}{1-\varepsilon}(1-v)\ \ \text{in}\ 
(0,\infty)\times\mathbb{R}^n, \\
v(0,\cdot)& = & \varepsilon\ \ \text{in}\ \mathbb{R}^n.
\end{array}\right.
\end{equation*}
Its solution is explicit, 
$$
v(t,x)=v(t)=1-(1-\varepsilon)e^{-\frac{f(\varepsilon)}{1-\varepsilon}t}.
$$
Since $v\rightarrow 1$ as $t\rightarrow +\infty$, we have that $u\rightarrow 1$
uniformly in $\mathbb{R}^n$ as $t\rightarrow +\infty$. 
Therefore, since $u(t,x)=\varphi(x+te)=u(T,x+(t-T)e)$, letting $T\to\infty$ we conclude
$u\equiv 1$.
\end{proof}

Next, we have to prove the convergence to 1 behind the front. 
Once we know that the solution remains larger than a small positive constant behind the front,   
the proof of the the convergence towards 1 is dimension independent.
We write this step in the following, which will be very useful also when
proving the precise level set  bounds of Theorem~\ref{LEVR}.

To simplify the proof, we assume the initial datum to belong to the domain $D_{u,b}(A)$.
The lemma, however, holds without this assumption thanks to the more involved
maximum principle of subsection~2.6; see Remark~\ref{without} below.

\begin{lemma}\label{CONV1}
Let $n\geq 1$, $\alpha\in(0,1)$, $f$ satisfy \eqref{nonl},
and $p$ be a kernel satisfying \eqref{P1}-\eqref{P2}-\eqref{P3}.
Let $u$ be a solution of 
\eqref{N} with  $0\leq u\leq 1$ such that $u(0,\cdot)\in D_{u,b}(A)$ and 
\begin{equation}\label{epsleverR}
u\geq\varepsilon \ \quad \text{ for all } t\geq t_0 \text{ and } |x|\leq a e^{\nu t},
\end{equation}
for some positive constants $\varepsilon\in(0,1)$, $a$, $\nu$, and $t_0$. Then, we have:

{\rm i)} For all $\lambda \in(0,1)$ there exist constants $t_\lambda >t_0$ and $C_\lambda >0$ 
such that 
\begin{equation}\label{alllevelR}
u\geq \lambda \ \quad \text{ for all } t\geq t_\lambda \text{ and } |x|\leq \frac{1}{C_\lambda}e^{\nu t}.
\end{equation}

{\rm ii)} For every $\sigma\in (0,\nu)$, $u(t,x)\to 1$ 
uniformly in $\left\{|x|\leq e^{\sigma t}\right\}$ as $t\to +\infty$.
\end{lemma}

Note that in \eqref{epsleverR} and \eqref{alllevelR} we have the same exponent $\nu$
in the exponential. This will be a key point to establish Theorem~\ref{LEVR}  concerning
the bounds on the level sets with the exact exponent (in the exponential).

\begin{remark}\label{without}
{\rm
The statement of Lemma~\ref{CONV1} will suffice for our purposes. However,
the lemma also holds without the assumption $u(0,\cdot)\in D_{u,b}(A)$. This
assumption on the initial datum being in the domain allows to use the simple maximum
principle of Proposition~\ref{nonlocMPdomain} and its immediate consequence:
Lemma~\ref{AR}.

The lemma holds without the assumption $u(0,\cdot)\in D_{u,b}(A)$ since we can apply
instead the maximum principle of Proposition~\ref{nonlocMPmild}, which gives that
Lemma~\ref{AR} also holds without the hypothesis on  $v(t,\cdot)\in D_\gamma(A)$ for all 
$t>0$.
}
\end{remark}

To prove Lemma~\ref{CONV1}, we need to use a comparison function modeled
by $w_\gamma(x)=|x|^{\gamma}$. Thus, we consider the semigroup in the space $X_\gamma$
introduced in subsection~2.1. To use the simple maximum principles of subsection~2.5
for classical solutions,
instead of using as initial datum $w_\gamma(x)=|x|^{\gamma}$ we use the function
\begin{equation}\label{deftildew}
\tilde{w}_\gamma(x)=\int_0^1T_s w_{\gamma}\,ds,
\end{equation}
which belongs to $D_\gamma(A)$ as pointed out in \eqref{intdomain}.

In addition, since $T_t\tilde{w}_\gamma(x)=\int_t^{t+1} T_s w_{\gamma}\,ds$, 
using the bounds in part b) of Lemma \ref{HR}, we deduce
\begin{equation}\label{uppgammadom}
T_t \tilde w_\gamma(x)\leq C_{\gamma}(|x|^{\gamma}+(t+1)^{\frac{\gamma}{2\alpha}})
\quad\text{ for all } t>0,\, 
x\in\mathbb{R}^n,
\end{equation}
and
\begin{equation}\label{lowgammadom}
T_t \tilde w_\gamma(x)\geq c_{\gamma}|x|^{\gamma}\quad\text{ if } t>0,\, 
|x|\geq (t+1)^{\frac{1}{2\alpha}}.
\end{equation}
The constants $C_{\gamma}$ and $c_\gamma$ depend only on $n$, $\alpha$, $B$, and $\gamma$.

\begin{proof}[Proof of Lemma \ref{CONV1}]
Since $u(0,\cdot)\in D_{u,b}(A)$, for any $\gamma\in (0,2\alpha)$ the mild solution $u$ satisfies
$u \in C^1([0,\infty);X_\gamma)$, $u([0,\infty))\subset D_{u,b}(A)\subset D_{\gamma}(A)$,
and it is a classical solution (see Remark~\ref{regularity}).
By hypothesis, for every $t_1\in[t_0,\infty)$ (to be chosen later),
\begin{equation}\label{epsregionR}
\varepsilon\leq u\leq 1 \ \quad \text{ in } \Omega_r:= \left\{t> t_1\ 
,|x|< \overline{r}(t):=ae^{\nu t}\right\}.
\end{equation}
Since $f$ is concave and $f(0)=f(1)=0$, for every $0<\varepsilon'<\varepsilon$ we have 
\begin{equation}\label{epsstrictR}
f(s)\geq \frac{f(\varepsilon')}{1-\varepsilon'}(1-s)\quad \ \text{ for all }\ \ s\in[\varepsilon,1].
\end{equation}
We take $\varepsilon'\in(0,\varepsilon)$ small enough so that
\begin{equation*}
0<q_{\varepsilon'}:=\frac{f(\varepsilon')}{1-\varepsilon'}<2\alpha\nu.
\end{equation*}
With this choice of $\varepsilon'$, we take $\gamma$ defined by
\begin{equation*}
0<\gamma:=\frac{q_{\varepsilon'}}{\nu}<2\alpha.
\end{equation*}
Note that by \eqref{epsregionR} and \eqref{epsstrictR}, we have 
\begin{equation}\label{epsineqR}
(\partial_t+A)(1-u)=-f(u)\leq -q_{\varepsilon'}(1-u)\ \ \text{ in } \Omega_r.
\end{equation}

We now use as comparison function the solution $w$ of
\begin{equation*}
\left \{ \begin{array}{rcll}
w_t+Aw & = & -q_{\varepsilon'}w&\text{in } [t_1,\infty)\times\mathbb{R}^n, \\
w(t_1,x)& = & 1+\frac{1}{c_{\gamma} a^{\gamma}}\tilde{w}_\gamma(x)& \text{for }
x\in\mathbb{R}^n,
\end{array}\right.
\end{equation*}
where $\tilde{w}_\gamma\in X_\gamma$ has been defined in \eqref{deftildew}.
Here, $a$ is the constant in 
\eqref{epsregionR} and $c_{\gamma}$ the constant in \eqref{lowgammadom}. 
The solution in the space $X_\gamma$ of this linear problem is given by
\begin{equation*}
w(t,x)=e^{-q_{\varepsilon'}(t-t_1)}\left\{1+\frac{1}{c_{\gamma} a^{\gamma}}
T_{t-t_1}\tilde{w}_\gamma(x)\right\}
\end{equation*}
for $t\geq t_1$ and $x\in\mathbb{R}^n$.
Since $\tilde{w}_\gamma \in D_\gamma(A)$, the solution $w$ is classical;
in particular, $w\in C^1([t_1,\infty);X_\gamma)$ and $w([t_1,\infty))\subset D_\gamma(A)$.

We apply Lemma \ref{AR} to 
$$
v:=(1-u)-w,
$$ 
with initial time $t_1$, $c(t,x)\equiv -q_{\eps'}$, and $|x|\leq\overline{r}(t):=ae^{\nu t}$ 
in \eqref{epsregionR}. We know that $v\in C^1([t_1,\infty);X_\gamma)$ and 
$v([t_1,\infty))\subset D_\gamma(A)$.

Condition \eqref{AR1} with 
$t=0$ replaced by $t=t_1$, i.e.,  $v\leq 0$ for $t=t_1$ in 
$\mathbb{R}^n$, holds since $1-u\leq1\leq w$ for $t=t_1$.

To verify \eqref{AR2}, we take $t_1\geq t_0$ large enough
to guarantee $ae^{\nu t}\geq (t+1)^{\frac{1}{2\alpha}}\geq (t-t_1+1)^{\frac{1}{2\alpha}}$ 
for $t\geq t_1$. 
Thus, the lower bound in \eqref{lowgammadom} gives that
if $t\geq t_1$ and $|x|\geq \overline{r}(t)$, then
$T_{t-t_1}\tilde w_\gamma(x)\geq c_\gamma |x|^\gamma\geq c_\gamma a^\gamma e^{\gamma\nu t}$.
Hence, 
$$
w(t,x)\geq e^{-q_{\varepsilon'}t}e^{q_{\varepsilon'}t_1}e^{\gamma\nu t}\geq 
e^{(\gamma\nu-q_{\varepsilon'})t}=1\geq1-u(t,x)\ \ \text{if } t\geq t_1 \text{ and }
|x|\geq\overline{r}(t).
$$

Finally, \eqref{AR3} clearly holds since, by \eqref{epsineqR},
$$v_t+Av=-f(u)+q_{\varepsilon'}w\leq -q_{\varepsilon'}(1-u-w)=-q_{\varepsilon'}v
\quad\text{ in } \Omega_r.
$$
Therefore, by Lemma \ref{AR}, $v\leq0$ in $[t_1,\infty)\times\mathbb{R}^n$
for some $t_1$ taken to be large enough. Thus, using also
the upper bound \eqref{uppgammadom}, we conclude
\begin{equation}\label{final}
\begin{split}
1-u (t,x)&\leq  w(t,x)=e^{-q_{\varepsilon'}(t-t_1)}\left\{1+\frac{1}{c_{\gamma} 
a^{\gamma}} T_{t-t_1}\tilde{w}_\gamma(x)\right\} \\ 
& \leq  e^{-q_{\varepsilon'}(t-t_1)} \left\{1+C_{a,\gamma}(|x|^{\gamma}+
(t-t_1+1)^{\frac{\gamma}{2\alpha}})\right\}\ \text{ in }\R^n, \text{ if } t\geq t_1,
\end{split}
\end{equation}
for some constant $C_{a,\gamma}$ depending on $a$ and $\gamma$.

{From} this bound, we deduce the two statements of the lemma. First, to prove part i), 
in the new region $\left\{t\geq t_\lambda,\, |x|\leq C_\lambda^{-1} 
e^{\nu t}\right\}$ (where $t_\lambda$ and $C_\lambda$ are to be chosen next), we have
\begin{eqnarray*}
(1-u)(t,x) &\leq & e^{-q_{\varepsilon'}(t-t_1)}\left\{1+C_{a,\gamma}
(C_\lambda^{-\gamma} e^{\gamma \nu t}+(t+1)^{\frac{\gamma}{2\alpha}})\right\}\\
 &=&e^{q_{\varepsilon'}t_1}\left\{e^{-q_{\varepsilon'}t}+C_{a,\gamma}C_\lambda^{-\gamma}+
C_{a,\gamma}(t+1)^{\frac{\gamma}{2\alpha}}e^{-q_{\varepsilon'}t}\right\}\\
 &\leq&\frac{1-\lambda}{2}+e^{q_{\varepsilon'}t_1}C_{a,\gamma}C_\lambda^{-\gamma}\leq
1-\lambda
\end{eqnarray*}
if we take both $t_\lambda$ and $C_\lambda$ large enough.
Thus, $u\geq \lambda$ in this region, as claimed.

Inequality \eqref{final} also shows part ii) of the lemma, that is,
the uniform convergence of $u$ towards $1$
in the region $\left\{|x|\leq e^{\sigma t}\right\}$ when $\sigma<\nu$.
Simply use that $\gamma \sigma <\gamma\nu=q_{\varepsilon'}$.
\end{proof}

We can finally establish our first main result.

\begin{proof}[Proof of Theorem \ref{LIMR}]
Part a) is simple. Since $f(s)\leq f'(0)s$ for all $s\in [0,1]$, we have that
$u\leq v$ where $v$ is the solution of 
$v_t+Av=f'(0)v$ with initial condition $u_0$. It is given by
$$
v(t,x)=e^{f'(0)t}\, T_tu_0(x).
$$
Since $u_0(x)\leq \min(1,C|x|^{-n-2\alpha})$, the upper bound in part a) of
Lemma \ref{HR} leads to $T_t u_0(x)\leq C t|x|^{-n-2\alpha}$ 
for $t\geq 1$ and $x\in\mathbb{R}^n$.
Thus, 
$$u(t,x)\leq v(t,x)\leq C t e^{f'(0)t}|x|^{-n-2\alpha}\ \quad\text{ for all } t\geq 1
\text{ and }x\in\mathbb{R}^n.
$$ 
{From} this, statement a) in the theorem follows immediately.
Indeed, for $|x|\geq e^{\sigma t}$ and $t$ large enough, we deduce
$$
u(t,x)\leq C t e^{f'(0)t}e^{-(n+2\alpha)\sigma t}\ \longrightarrow\ 0 
\ \ \text{as}\ \ t\uparrow\infty,
$$
since $\sigma > f'(0)/(n+2\alpha)$.

To prove part b) of the theorem, note that it suffices to establish it for the solution 
of \eqref{N} with a smaller
initial datum that $u(2,\cdot)$, i.e., $u$ at time $2$. We replace $u(2,\cdot)$ at time $2$
by the smaller initial datum $\underline{u}_0:=c\int_1^2 p(s,\cdot) ds$.
By Lemma~\ref{appR}, $\underline{u}_0\leq T_2u_0 \leq u(2,\cdot)$ and hence, 
$\underline{u}(t,\cdot)\leq u(t+2,\cdot)$ for all $t>0$, where $\underline{u}$ is the solution with
initial datum $\underline{u}_0$. In addition, by the same lemma,
$\underline{u}_0\in D_0(A)\subset D_{u,b}(A)$, 
and this will allow us to apply Lemma~\ref{CONV1} to $\underline{u}$. 
Now, given $\sigma <\sigma_\ast$, take $\sigma'$ such that
$$
0<\sigma<\sigma'<\frac{f'(0)}{n+2\alpha}.
$$
We first apply Corollary \ref{BDR} to $\underline{u}$ with $\sigma$ replaced by $\sigma'$. We obtain
\begin{equation*}
\underline{u}\geq \varepsilon\ \quad  \text{ if } t\geq t_0\ ,|x|\leq be^{\sigma't},
\end{equation*}
for some constants $b>0$ and $t_0$. 
Hence, we can apply Lemma \ref{CONV1} to  $\underline{u}$ 
with $\nu$ replaced by $\sigma'$. Part ii) of the lemma gives the desired convergence
of $\underline{u}$ (and hence of $u$) towards $1$.
\end{proof}

\section{Nondecreasing initial data}

The plan is the same as that of Section 3.
To prove part b) of Theorem~\ref{LIMI}, we need a key lemma similar to Lemma \ref{EXPR}.

\begin{lemma}\label{EXPI}
Let $n= 1$, $\alpha\in(0,1)$, $f$ satisfy \eqref{nonl},
and $p$ be a kernel satisfying \eqref{P1}-\eqref{P2}-\eqref{P3}.
Recall that $B$ is the constant in \eqref{P3}. 
Then, for every 
$0<\sigma<\frac{f'(0)}{2\alpha}$, there exist $t_0\geq1$ and 
$0<\varepsilon_0 <1$ 
depending only on $\alpha$, $B$, $f$, and $\sigma$, for which the following 
holds. 

Given $x_0\leq-1$ and $0<\varepsilon\leq\varepsilon_0$, let $a_0>0$ be defined by 
$a_0|x_0|^{-2\alpha}=\varepsilon$, and let
\begin{equation*} 
V_0(x)=\left \{ \begin{array}{lll}
a_0 |x|^{-2\alpha}   &\text{for} &x\leq x_0, \\
\varepsilon = a_0 |x_0|^{-2\alpha} &\text{for} &x\geq x_0.
\end{array}\right.
\end{equation*}
Then, the mild solution $v$ of \eqref{N} with initial condition $V_0$ satisfies
$$v(kt_0,x)\geq \varepsilon \ \ \ \text{for } x\geq x_0e^{\sigma kt_0}
$$
and $k\in\left\{0,1,2,3,\ldots\right\}$.
\end{lemma}

\begin{proof}
The result being true  for $k=0$, let us  prove it for $k=1$. Let $\delta\in(0,1)$ be 
sufficiently small such that
\begin{equation}\label{delta}
\sigma<\frac{1}{2}\left(\sigma+\frac{f'(0)}{2\alpha}\right)<\frac{1}{2\alpha}
\frac{f(\delta)}{\delta}<\frac{1}{2\alpha}f'(0).
\end{equation}
We take $t_0\geq 1$ sufficiently large, 
depending only on 
$\alpha$, $B$, $f$ and $\sigma$, such that
\begin{equation}\label{deft0I}
\left(c\frac{t_0}{t_0^{\frac{1}{2\alpha}+1}+1}\right)^\frac{1}{2\alpha} 
e^{\frac{1}{2}\left(\sigma+\frac{f'(0)}{2\alpha}\right)t_0}
\geq e^{\sigma t_0},
\end{equation}
where $c>0$ is the constant in the lower bound in part~a) of Lemma~\ref{HI}.
In particular, $c$ depends only on $\alpha$ and $B$.
Define now $0<\varepsilon_0<\delta$ by
\begin{equation*}
\varepsilon_0 = \delta e^{-f'(0)t_0}.
\end{equation*}

Recall that, in what follows, we are given $x_0\leq-1$ and $\varepsilon$ such that
$$
0<\varepsilon\leq\varepsilon_0<\delta.
$$
Let 
$$
w:=e^{(f(\delta)/\delta)t}\,T_tV_0. 
$$
It satisfies 
$$
w_t+Aw=\frac{f(\delta)}{\delta}w ,\ \ w(0,\cdot)=V_0
$$
in the mild sense. Since $V_0\leq\varepsilon$ in $\R$, we also have
$T_tV_0\leq\varepsilon$ in $\R$ for all $t>0$.
Now, for $t\leq t_0$, $0\leq w\leq  e^{(f(\delta)/\delta)t_0}\varepsilon\leq 
e^{f'(0)t_0}\varepsilon_0=\delta$.
Since $(f(\delta)/\delta)w\leq f(w)$ for $0\leq w\leq\delta$, 
we have that $w$ is a mild subsolution of \eqref{N} in $\left[0,t_0\right]\times\mathbb{R}$. 
Thus, $v(t_0,\cdot)\geq w(t_0,\cdot)$ in $\mathbb{R}$.
By the lower bound in part~a) of Lemma~\ref{HI}, we have
\begin{equation}\label{pow}
v(t_0,x)\geq w(t_0,x)\geq e^{(f(\delta)/\delta)t_0}c\frac{t_0}{t_0^{\frac{1}{2\alpha}+1}+1}
\frac{a_0}{|x|^{2\alpha}}\qquad \text{for}\ x\leq x_0. 
\end{equation}
Let us define $x_1<0$ by 
\begin{equation}\label{defx1}
e^{(f(\delta)/\delta)t_0}c\frac{t_0}{t_0^{\frac{1}{2\alpha}+1}+1}
\frac{a_0}{|x_1|^{2\alpha}}=\varepsilon.
\end{equation}
Since $a_0=\varepsilon|x_0|^{2\alpha}$, we get
$$
x_1=x_0\left(c\frac{t_0}{t_0^{\frac{1}{2\alpha}+1}+1}\right)^{\frac{1}{2\alpha}} 
e^{\frac{1}{2\alpha}\frac{f(\delta)t_0}{\delta}}.
$$
By \eqref{deft0I} and the second inequality in \eqref{delta}, we have
\begin{equation}\label{bddx1}
x_1\leq x_0 e^{\sigma t_0}<x_0.
\end{equation}

Now, since $x_1<x_0$, \eqref{pow} and \eqref{defx1} lead to $v(t_0,x)\geq a_1
|x|^{-2\alpha}$ for $x\leq x_1$, where 
$a_1:=\varepsilon \left|x_1\right|^{2\alpha}$. Since $v$ is nondecreasing in $x$
(see the last comment in subsection 2.4), 
we also have $v(t_0,x)\geq a_1
|x_1|^{-2\alpha}=\varepsilon$ for $x\geq x_1$. 

Thus, $v(t_0,\cdot)\geq V_1$ 
where $V_1$ is given by the expression for $V_0$ in the statement of the lemma 
with $(x_0,a_0)$ replaced by $(x_1,a_1)$. Note that $x_1\leq x_0\leq -1$.

Therefore, we can repeat the argument above successively, now with initial times $t_0, 2t_0,
3t_0,\ldots$  and points $x_1, x_2, x_3,\ldots$, and get that 
$$
v(kt_0,x)\geq\varepsilon\ \ \ \text{for } x\geq x_k
$$
for all $k\in\left\{0,1,2,3,\ldots\right\}$. Since
$$
x_k\leq x_0 e^{\sigma kt_0}
$$
by \eqref{bddx1}, the statement of the lemma follows.
\end{proof}

\begin{corollary}\label{BDI}
Let $n=1$, $\alpha\in(0,1)$, $f$ satisfy \eqref{nonl},
$p$ be a kernel satisfying \eqref{P1}-\eqref{P2}-\eqref{P3}, and
$0<\sigma<\frac{f'(0)}{2\alpha}$. Let $t_0\geq 1$ be the time given by Lemma~\ref{EXPI}. 

Then, for every measurable nondecreasing initial datum $u_0$ with $0\leq u_0\leq1$ and $u_0\not\equiv 0$, 
there exist $\varepsilon\in(0,1)$ 
and $b>0$ (both depending on $u_0$)
such that 
$$
u(t,x)\geq\varepsilon \quad\ \text{ for all } t\geq t_0 \text{ and }
x\geq -be^{\sigma t},
$$
where $u$ is the mild solution of $\eqref{N}$ with $u(0,\cdot)=u_0$.
\end{corollary}

\begin{proof}
Since $u$ is a supersolution of the homogeneous problem (the problem with $f=0$)
and $u_0\not\equiv0$, we have 
that $u(t_0/2,\cdot)\geq T_{t_0/2}u_0 > 0$ in $\mathbb{R}$. 
Since $u$ is nondecreasing in $x$ (see subsection~2.4), 
$u(t_0/2,x)\geq u(t_0/2,0)\geq T_{t_0/2} u_0(0)=:\eta>0$ for all $x\geq 0$
(recall that $T_{t_0/2}u_0$ is a continuous positive function).
Thus, $u(t_0/2,\cdot)\geq \eta \chi_{(0,\infty)}$ in $\mathbb{R}$,
for some constant $\eta>0$. The second inequality in \eqref{lowercharI} now gives, for
$t>0$ and $x\leq 0$, 
\begin{equation*}
u(t_0/2+t,x) \geq \eta T_t \chi_{(0,\infty)}(x) 
\geq \eta B^{-1} c(1+t^{-\frac{1}{2\alpha}}|x|)^{-2\alpha}.
\end{equation*}
We deduce
$$
u(t_0/2+t,x)\geq a_0|x|^{-2\alpha}\ \ \text{for} \ 
t\in\left[t_0/2,3t_0/2\right]\ \text{and}\ x\leq x_0:=-t_0^{\frac{1}{2\alpha}}\leq -1,
$$
for some $a_0>0$.
We make $a_0$ smaller, if necessary, to have that $\varepsilon:=a_0|x_0|^{-2\alpha}
\leq \varepsilon_0$, where $\varepsilon_0$ is given by Lemma \ref{EXPI}. 
Since $u$ is nondecreasing,  we deduce 
$$
u(t_0/2+t,\cdot)\geq V_0\ \ \text{ in }\mathbb{R} 
\text{ for all } t\in\left[t_0/2,3t_0/2\right],
$$
where $V_0$ is the initial condition in Lemma \ref{EXPI}. 

Thus, we can apply Lemma \ref{EXPI} to get a lower bound for $u(\cdot+\tau_0,\cdot)$ 
for all $\tau_0\in[t_0,2t_0]$. Since $\left\{\tau_0+kt_0\, |\, k=0,1,2,\ldots\ \ 
\text{and} \ \tau_0\in[t_0,2t_0]\right\}$ cover all $[t_0,\infty)$, we deduce 
$$
u(t,x)\geq\varepsilon\ \ \text{if}\ \ t\geq t_0\ \ \text{and}\ \ 
x\geq x_0 e^{-\sigma2t_0}e^{\sigma t}
$$
by taking $t=\tau_0+kt_0$ and using (recall here that $x_0<0$)
that $x\geq x_0 e^{-\sigma2t_0} e^{\sigma t}\geq 
x_0 e^{-\sigma\tau_0}e^{\sigma t}=x_0e^{\sigma kt_0}$. This last statement proves the corollary 
taking $b=|x_0|e^{-\sigma 2t_0}$. 
\end{proof}

We can now give the proof of Theorem \ref{LIMI}. 
Note that the previous lemma and corollary are crucial to guarantee that 
$u\geq\varepsilon$ for $x\geq -be^{\sigma t}$. Thus, in this region $f(u)$ is 
greater than a positive linear function vanishing at $u=1$. This will 
lead to the exponential convergence to 1 in the region. 

To show this and prove part b) of Theorem \ref{LIMI}, we need to use a comparison function modeled
by $W_\gamma(x)=(x_-)^{\gamma}$. Thus, we consider the semigroup in the space $X_\gamma$
introduced in subsection~2.1. To use the simple maximum principles of subsection~2.5
for classical solutions,
instead of using as initial datum $W_\gamma(x)=(x_-)^{\gamma}$ we use the function
\begin{equation}\label{deftildeW}
\tilde{W}_\gamma(x)=\int_0^1T_s W_{\gamma}\,ds,
\end{equation}
which belongs to $D_\gamma(A)$ as pointed out in \eqref{intdomain}.

In addition, since $T_t\tilde{W}_\gamma(x)=\int_t^{t+1} T_s W_{\gamma}\,ds$, 
using the bounds in part b) of Lemma \ref{HI}, we deduce
\begin{equation}\label{uppgammadomI}
T_t \tilde W_\gamma(x)\leq C_{\gamma}(|x|^{\gamma}+(t+1)^{\frac{\gamma}{2\alpha}})
\quad\text{ for all } t>0,\, x<0,
\end{equation}
and
\begin{equation}\label{lowgammadomI}
T_t \tilde W_\gamma(x)\geq c_{\gamma}|x|^{\gamma}\quad\text{ for all } t>0,\, x<0.
\end{equation}
The constants $C_{\gamma}$ and $c_\gamma$ depend only on $\alpha$, $B$, and $\gamma$.

\begin{proof}[Proof of Theorem \ref{LIMI}]
Part a) is simple. Since $f(s)\leq f'(0)s$ for all $s\in [0,1]$, we have that
$u\leq v$ where $v$ is the solution of 
$v_t+Av=f'(0)v$ with initial condition $u_0$. It is given by
$$
v(t,x)=e^{f'(0)t}\, T_tu_0(x).
$$
We know that $u_0(x)\leq C_0|x|^{-2\alpha}$ for some constant $C_0$; we may assume $C_0>1$.
Taking $x_0:=-C_0^{1/(2\alpha)}<-1$, we have $u_0\leq 1= C_0|x_0|^{-2\alpha}$
and thus $u_0\leq V_0$ in $\R$, where $V_0$ is the function in part a) of
Lemma \ref{HI}. The upper bound in part a) of
Lemma \ref{HI} leads to $T_t u_0(x)\leq C t|x|^{-2\alpha}$ 
for $t\geq 1$ and $x<2x_0$. Thus, 
$$
u(t,x)\leq v(t,x)\leq C t e^{f'(0)t}|x|^{-2\alpha}
$$
for $t\geq 1$ and $x<2x_0$. 
From this bound, statement a) in the theorem follows immediately. 
Indeed, for $x\leq-e^{\sigma t}$ and $t$ large enough, we deduce
$$
u(t,x)\leq C t e^{f'(0)t}e^{-2\alpha\sigma t}\ \longrightarrow\ 0 
\ \ \text{as}\ \ t\uparrow\infty,
$$
since $\sigma > f'(0)/(2\alpha)$. 

To prove part b) of the theorem, note that it suffices to establish it for the solution 
of \eqref{N} with a smaller
initial datum that $u(2,\cdot)$, i.e., $u$ at time $2$. We replace $u(2,\cdot)$
at time $2$ by the smaller initial datum $\underline{u}_0:=c\int_1^2 P(s,\cdot) ds$.
By Lemma~\ref{appI}, $\underline{u}_0\leq T_2u_0 \leq u(2,\cdot)$ and hence, 
$\underline{u}(t,\cdot)\leq u(t+2,\cdot)$ for all $t>0$, where $\underline{u}$ is the solution with
initial datum $\underline{u}_0$. In addition, by the same lemma,
$\underline{u}_0\in  D_{u,b}(A)$, 
and this will allow us to apply Lemma~\ref{AI} to $\underline{u}$. 
To simplify notation, in the rest of the proof we denote the solution $\underline{u}(t,\cdot)$
by $u(t,\cdot)$.

Since now $u(0,\cdot)\in D_{u,b}(A)$, the mild solution $u$ satisfies
$u \in C^1([0,\infty);X_\gamma)$ and $u([0,\infty))\subset D_{u,b}(A)\subset D_{\gamma}(A)$
for any $\gamma\in (0,2\alpha)$,
and it is a classical solution (see Remark~\ref{regularity}).

Now, given $\sigma <\sigma_{\ast\ast}$, take $\sigma'$ such that
$$
0<\sigma<\sigma'<\frac{f'(0)}{2\alpha}.
$$
We apply Corollary \ref{BDI} to $u$ with $\sigma$ replaced by $\sigma'$.
We obtain, for any $t_1\geq t_0$ ($t_0$ is given by the corollary), 
\begin{equation}\label{regionI}
\varepsilon\leq u \leq 1 \ \text{ in } \Omega:=\left\{t> t_1,x>\overline{x}(t):=
-be^{\sigma't}\right\},
\end{equation}
for some positive constants $\varepsilon$ and $b$.
Since $f$ is concave and $f(0)=f(1)=0$, for every $0<\varepsilon'<\varepsilon$ we have 
\begin{equation}\label{epsstrictI}
f(s)\geq \frac{f(\varepsilon')}{1-\varepsilon'}(1-s)\quad \ \text{ for all } s\in[\varepsilon,1].
\end{equation}
We take $\varepsilon'\in(0,\varepsilon)$ small enough so that
\begin{equation*}
0<q_{\varepsilon'}:=\frac{f(\varepsilon')}{1-\varepsilon'}<2\alpha\sigma'.
\end{equation*}
With this choice of $\varepsilon'$, we take $\gamma$ defined by
\begin{equation*}
0<\gamma:=\frac{q_{\varepsilon'}}{\sigma'}<2\alpha.
\end{equation*}
Note that by \eqref{regionI} and \eqref{epsstrictI}, we have 
\begin{equation}\label{epsineqI}
(\partial_t+A)(1-u)=-f(u)\leq -q_{\varepsilon'}(1-u)\ \ \text{ in } \Omega.
\end{equation}

We now use as comparison function the solution $w$ of
\begin{equation*}
\left \{ \begin{array}{rcll}
w_t+Aw & = & -q_{\varepsilon'}w&\text{in } [t_1,\infty)\times\mathbb{R}, \\
w(t_1,x)& = & 1+\frac{1}{c_{\gamma} b^{\gamma}}\tilde{W}_\gamma(x)& \text{for }
x\in\mathbb{R},
\end{array}\right.
\end{equation*}
where $\tilde{W}_\gamma\in X_\gamma$ has been defined in \eqref{deftildeW}.
Here, $b$ is the constant in 
\eqref{regionI} and $c_{\gamma}$ the constant in \eqref{lowgammadomI}. 
The solution in the space $X_\gamma$ of this linear problem is given by
\begin{equation*}
w(t,x)=e^{-q_{\varepsilon'}(t-t_1)}\left\{1+\frac{1}{c_{\gamma} b^{\gamma}}
T_{t-t_1}\tilde{W}_\gamma(x)\right\}
\end{equation*}
for $t\geq t_1$ and $x\in\mathbb{R}$.
Since $\tilde{W}_\gamma \in D_\gamma(A)$, the solution $w$ is classical;
in particular, $w\in C^1([t_1,\infty);X_\gamma)$ and $w([t_1,\infty))\subset D_\gamma(A)$.

We apply Lemma \ref{AI} to 
$$
v:=(1-u)-w,
$$ 
with initial time $t_1$, $c(t,x)\equiv -q_{\eps'}<0$, and $\overline{x}(t):=-be^{\sigma' t}$ 
in \eqref{regionI}. We know that $v\in C^1([t_1,\infty);X_\gamma)$ and 
$v([t_1,\infty))\subset D_\gamma(A)$.

Condition \eqref{AI1} with 
$t=0$ replaced by $t=t_1$, i.e.,  $v\leq 0$ for $t=t_1$ in 
$\mathbb{R}$, holds since $1-u\leq1\leq w$ for $t=t_1$.

To verify \eqref{AI2}, we use the lower bound in \eqref{lowgammadomI}.
For $t\geq t_1$ and $x\leq \overline{x}(t)<0$, we have  
$T_{t-t_1}\tilde W_\gamma(x)\geq c_\gamma |x|^\gamma\geq c_\gamma b^\gamma e^{\gamma\sigma' t}$.
Hence, 
$$
w(t,x)\geq e^{-q_{\varepsilon'}t}e^{q_{\varepsilon'}t_1}e^{\gamma\sigma' t}\geq 
e^{(\gamma\sigma'-q_{\varepsilon'})t}=1\geq1-u(t,x)\ \ \text{if } t\geq t_1 \text{ and }
x\leq\overline{x}(t).
$$

To verify \eqref{AI3}, we use Proposition~\ref{limitsinfincr}.
Let $l:=\lim_{x\to +\infty} u(t_1,x)$. Since $\phi_l(t)$ is nondecreasing in $t$,
the proposition gives that $\limsup_{x\to +\infty} (1-u)(t,x)=1-\phi_l(t)
\leq 1-\phi_l(t_1)=:\delta$ uniformly in $t\in [t_1,T]$ for all $T>t_1$.
We apply Lemma~\ref{AI} with this choice of $\delta$.

Finally, \eqref{AI4} clearly holds since, by \eqref{epsineqI},
$$
v_t+Av=-f(u)+q_{\varepsilon'}w\leq -q_{\varepsilon'}(1-u-w)=-q_{\varepsilon'}v
\quad\text{ in } \Omega.
$$

Therefore, by Lemma \ref{AI}, for all $t_1\geq t_0$ we have $v\leq\delta =1-\phi_l(t_1)$ in 
$[t_1,\infty)\times\mathbb{R}$. Thus, using 
the upper bound \eqref{uppgammadomI}, we conclude
\begin{equation*}
\begin{split}
1-u (t,x)&\leq 1-\phi_l(t_1)+ w(t,x)\\
&= 1-\phi_l(t_1)+
e^{-q_{\varepsilon'}(t-t_1)}\left\{1+\frac{1}{c_{\gamma} 
b^{\gamma}} T_{t-t_1}\tilde{W}_\gamma(x)\right\} \\ 
& \leq 1-\phi_l(t_1)+ 
e^{-q_{\varepsilon'}(t-t_1)} \left\{1+C_{b,\gamma}(|x|^{\gamma}+
(t-t_1+1)^{\frac{\gamma}{2\alpha}})\right\}
\end{split}
\end{equation*}
if $t\geq t_1$ and $x<0$, for some constant $C_{b,\gamma}$ depending only on
$\alpha$, $B$, $b$ and $\gamma$.

This inequality shows part b) of the theorem, that is,
the uniform convergence of $u$ towards $1$
in the region $\left\{x\geq -e^{\sigma t}\right\}$. Indeed, given $\eps>0$ choose $t_1\geq t_0$
large enough such that $1-\phi_l(t_1)<\eps$;  
recall that the solution of the ODE, $\phi_l(t)$, tends to $1$ as $t\to\infty$.
With this choice of $t_1$, the remaining term of the above bound is also smaller than
$\eps$ for $t$ large enough; 
simply use that $\gamma \sigma <\gamma\sigma' =q_{\varepsilon'}$.
This ends the proof of Theorem \ref{LIMI}. 
\end{proof}

\section{Level set bounds in $\mathbb{R}$ when $A=(-\Delta)^{1/2}$ }

In this section we consider $n=1$, $A=(-\Delta)^{1/2}$, and $f(u)=u(1-u)$, that is, equation
\begin{equation} \label{N12bis}
u_t + (-\Delta)^{1/2} u =  u(1-u) \ \ \ \ \text{in} \ 
\left(0,+\infty\right)\times \mathbb{R}.
\end{equation}

The transition kernel $p_{1/2}$ is known explicit, even in dimension $n$. It is given by
$p_{1/2}(t,x)=B_nt^{-n}\left(1+t^{-2}r^2\right)^{-\frac{n+1}{2}}=
B_nt(t^2+r^2)^{-\frac{n+1}{2}}$, where $r=|x|$ and $B_n=\Gamma(\frac{n+1}{2})\pi^{-\frac{n+1}{2}}$ 
is a positive constant. Thus, we have 
\begin{eqnarray*}
(-\Delta)^{1/2}p_{1/2}&=&-\partial_tp_{1/2}\\
&=&B_n\left\{nt^{-n-1}\left(1+t^{-2}r^2\right)^{-\frac{n+1}{2}}-(n+1)t^{-n}
\left(1+t^{-2}r^2\right)^{-\frac{n+3}{2}}t^{-3}r^2\right\}\\
&=&B_nt^{-n-1}\left(1+t^{-2}r^2\right)^{-\frac{n+3}{2}}\left\{n
\left(1+t^{-2}r^2\right)-(n+1)t^{-2}r^2\right\}\\
&=&B_nt^{-n}t^{-1}\left(1+t^{-2}r^2\right)^{-\frac{n+3}{2}}\left\{n-t^{-2}r^2\right\}.
\end{eqnarray*}
{From} this we deduce that, given a constant $b>0$,
\begin{equation} \label{lap12}
(-\Delta)^{1/2}\left(1+b^{-2}r^2\right)^{-\frac{n+1}{2}} =
b^{-1}\left(1+b^{-2}r^2\right)^{-\frac{n+3}{2}}\left\{n-b^{-2}r^2\right\}\quad\text{ in }\R^n.
\end{equation}

Consider now, on the model of $p_{1/2}$, a function $u$ of the form
$$
u(t,x)=a\left(1+\frac{r^2}{b(t)^2}\right)^{-\frac{n+1}{2}}
$$
with $b=b(t)$ to be chosen later.
Using \eqref{lap12}, we compute $u_t+(-\Delta)^{1/2}u-u(1-u)$ in $\mathbb{R}^n$:
\begin{eqnarray*}
u_t &=& a \left(1+b^{-2}r^2\right)^{-\frac{n+3}{2}}(n+1)b^{-3}b'r^2,\\
(-\Delta)^{1/2}u &=& a\left(1+b^{-2}r^2\right)^{-\frac{n+3}{2}}b^{-1}
\left(n-b^{-2}r^2\right), \\
u(1-u) &=& a\left(1+b^{-2}r^2\right)^{-\frac{n+1}{2}}
\left\{1-a\left(1+b^{-2}r^2\right)^{-\frac{n+1}{2}}\right\} \\
      &=& a\left(1+b^{-2}r^2\right)^{-\frac{n+3}{2}}
\left\{1+b^{-2}r^2-a\left(1+b^{-2}r^2\right)^{-\frac{n-1}{2}}\right\}.
\end{eqnarray*}
Thus, we have
\begin{eqnarray}
&& a^{-1}\left(1+b^{-2}r^2\right)^{\frac{n+3}{2}}\left\{ u_t+(-\Delta)^{1/2}u-(u-u^2)\right\} =
\nonumber\\
&& \hspace{1cm}= nb^{-1}-1 +a \left(1+b^{-2}r^2\right)^{-\frac{n-1}{2}}  +b^{-3}r^2
\left\{(n+1)b'-1-b\right\}.
\label{comp12}
\end{eqnarray}
We wish the above function $u$ to serve as a sub or a supersolution depending on its 
parameters. We have:

\begin{lemma}\label{SUBSUP}
Let $n=1$. For $a>0$ and $b_0>1$, let  
$$
u_{a,b_0}(t,x):=a\left(1+\frac{x^2}{\left\{(1+b_0)e^{t/2}-1\right\}^2}\right)^{-1}
\quad\text{ for }t>0,x\in\R.
$$ 
Then,

{\rm a)} If $a\leq\di\frac{b_0-1}{b_0}$, then $u_{a,b_0}$ is a subsolution of \eqref{N12bis}. 

{\rm b)} If $a\geq 1$, then  $u_{a,b_0}$ is a supersolution of \eqref{N12bis}.
\end{lemma}

\begin{proof}
Let $b(t)=(1+b_0)e^{t/2}-1$. Note that $2b'(t)=(1+b_0)e^{t/2}=1+b(t)$. Thus, by \eqref{comp12},
\begin{equation*}
a^{-1}\left(1+b(t)^{-2}x^2\right)^2\left\{ u_t+(-\Delta)^{1/2}u-(u-u^2)\right\} =
b(t)^{-1}-1 +a.
\end{equation*}

Now, since $b(t)\geq b_0$ for all $t>0$, the last expression satisfies
$b(t)^{-1}-1 +a\leq b_0^{-1}-1 +a=a-\frac{b_0-1}{b_0}\leq 0$ under the assumption in part a). 

Finally, since $b(t)^{-1}-1 +a\geq -1 +a\geq 0$ under the assumption in part b).
\end{proof}

Using this result and also our key Lemma~\ref{CONV1}, we can finally give the

\begin{proof}[Proof of Theorem \ref{LEVR}]
Let $\lambda\in (0,1)$. We start proving the inclusion
\begin{equation*}
\{|x|>C_\lambda e^{t/2} \}\subset \{u<\lambda\} \quad\text{ for all } t>0
\end{equation*}
if $C_\lambda$ is chosen large enough.
We simply use the explicit supersolution $u_{a,b_0}$ of Lemma~\ref{SUBSUP}
for some appropriate $a\geq 1$ and $b_0>1$. Take it at time $t=0$: 
$$
u_{a,b_0}(0,x)=a\left( 1+\frac{x^{2}}{b_0^2} \right)^{-1}\geq
\left( 1+\frac{x^{2}}{b_0^2} \right)^{-1}
\geq \frac{b_0^{2}}{2}|x|^{-2}  \quad\text{ if }|x|\geq b_0.
$$
Recall that we assume $u_0(x)\leq C|x|^{-2}$. Thus $u_0\leq u_{a,b_0}(0,\cdot)$
for $|x|\geq b_0$ if we take $b_0>1$ large enough (independently of $a\geq 1$, that
we can still choose). Now, by taking $a\geq 1$ large enough we also have
$u_0\leq u_{a,b_0}(0,\cdot)$ in $\{|x|\leq b_0\}$, and hence in all of $\R$.

We apply the comparison principle of subsection~2.4. Since 
$0\leq u\leq 1$, $0\leq u_{a,b_0}\leq a$, and $a\geq 1$,
here we change $f$ given by $f(u)=u-u^2$ outside $[0,a]$ to have
hypothesis \eqref{hypf} on the new $f$. The comparison principle gives that
$u(t,x)\leq u_{a,b_0}(t,x)$ for all $(t,x)$, that is,
$$
u(t,x) \leq a\left(1+\frac{x^2}{\left( (1+b_0)e^{t/2}-1\right)^2}\right)^{-1}
$$
Hence, if $u(t,x)\geq\lambda$ then 
$$
1+\frac{x^2}{\left( (1+b_0)e^{t/2}-1\right)^2}\leq\frac{a}{\lambda}
$$
and thus
$|x|\leq(1+b_0)\sqrt{a/\lambda}\, e^{t/2}$.

Next, we prove the other inclusion in \eqref{incl}:
\begin{equation}\label{incl2}
\{|x|<\frac{1}{C_\lambda} e^{t/2} \}\subset \{u>\lambda\} \quad\text{ for } t>t_\lambda,
\end{equation}
if $t_\lambda$ and $C_\lambda$ are chosen large enough.
Clearly, it suffices to prove this statement for the solution of \eqref{N} with a smaller
initial datum that $u(2,\cdot)$, i.e., $u$ at time $2$. We replace $u(2,\cdot)$
by the smaller initial datum $\underline{u}_0:=c\int_1^2 p(s,\cdot) ds$ at time $2$.
By Lemma~\ref{appR}, $\underline{u}_0\leq T_2u_0\leq u(2,\cdot)$ and hence, 
$\underline{u}(t,\cdot)\leq u(t+2,\cdot)$ for all $t>0$, where $\underline{u}$ is the solution with
initial datum $\underline{u}_0$. In addition, by the same lemma,
$\underline{u}_0\in D_0(A)\subset D_{u,b}(A)$, 
and this will allow us to use Lemma~\ref{CONV1} to $\underline{u}$.
To simplify notation, we denote $\underline{u}(t,\cdot)$ again by $u(t,\cdot)$.

Now we use crucially Lemma~\ref{CONV1} with $\nu=1/2$ in its statement. 
It requires the initial datum to belong
to the domain, as we have in the present situation. It gives that \eqref{incl2} will hold for every
$\lambda\in (0,1)$ (for some $t_\lambda$ depending on $\lambda$)
once we have proved it for one level set $\lambda=\varepsilon\in (0,1)$.
Hence, we can choose $\lambda=\varepsilon$ as small as needed in  \eqref{incl2}.

Note that Corollary~\ref{BDR} gives the analogue of \eqref{incl2} with 
$e^{t/2}$ replaced by $e^{\sigma t}$ for every $\sigma<1/2$ (and some $\lambda=\varepsilon$
small enough). To prove \eqref{incl2} with $\sigma =1/2$ we need to be 
more precise and we use a subsolution
from Lemma~\ref{SUBSUP}.

Since $u(1,\cdot)>0$ is a positive continuous function in all of $\R$, it is larger than a small 
positive constant times
the characteristic function of the unit interval. Thus, \eqref{lowercharR} applied 
with initial time~$1$ gives
$$
u(t,x)\geq 4c \frac{1}{(t-1)\{1+(t-1)^{-2}x^2\}}\ \ \quad\text {for all }t>1,|x|>1,
$$
for some constant $c>0$ depending on $u_0$. Now, since $t-1\geq t/2$ for $t\geq 2$, we have that 
$u(t,x)\geq 4c /(t\{1+(t-1)^{-2}x^2\})\geq c /(t\{1+t^{-2}x^2\})$ for all $t\geq 2$ and $|x|>1$.
Therefore, for all $T\geq 2$ we have
\begin{equation}\label{baix}
u(T,x)\geq \frac{c}{T}\, \frac{1}{1+T^{-2}x^2}\quad\text{ for all }x\in\R,
\end{equation}
for some positive constant $c=c(T)$ (depending on $T$ and $u_0$) taken to be small enough
to guarantee \eqref{baix} also for $|x|\leq 1$.
Taking $c$ smaller if necessary, we may assume 
$$
c<T-1.
$$

{From} now on we fix one time $T\geq 2$ and the constant $c=c(T)$ in \eqref{baix}.
We could take $T=2$ for instance. We place a subsolution  
$u_{a,b_0}(0,\cdot)$ of Lemma \ref{SUBSUP} below $u(T,\cdot)$. Note here the difference of times, $0$
and $T$, for both functions.  
We simply take $a=\frac{c}{T}$ and $b_0=T$. Since $a=\frac{c}{T} <\frac{T-1}{T}=\frac{b_0-1}{b_0}$,
we have that
$u_{c/T,T}$ is a subsolution. Note that 
$$
u(T,x)\geq   \frac{c}{T}\, \frac{1}{1+T^{-2}x^2}=u_{c/T,T}(0,x) \quad\text{ for all }x\in\R
$$
thanks to \eqref{baix}. Thus, for $t\geq T$ and all $x\in\R$, we have
$$
u(t,x)\geq u_{c/T,T}(t-T,x)=\frac{c/T}{1+\frac{x^2}
{\{(1+T) e^{(t-T)/2}-1\}^2}}.
$$
Hence, if $|x|\leq e^{t/2}$ and $t$ is large enough, we have
$u(t,x)>\varepsilon$
for $t$ large enough, for some constant $\varepsilon>0$.
\end{proof}

\end{document}